\documentclass[11pt]{amsart}
\usepackage[margin=1in]{geometry}
\usepackage{amsmath}
\usepackage{amsfonts,amsmath}
\usepackage{paralist}
\usepackage[colorlinks=true]{hyperref}
\hypersetup{urlcolor=blue, citecolor=red}
 \numberwithin{equation}{section}
\begin{document}
	\newcommand{\oo}{\overline{\Omega}}
	\newcommand{\thl}{\theta_{1,L}(|\m|^2)}
	\newcommand{\imm}[1]{-\textup{div}\left[(I+\m\otimes\m)\nabla{#1}\right]}
	\newcommand{\dy}{\textup{dist}(y,\partial\Omega)}
	\newcommand{\ml}[2]{\left[{#1}\right]_{#2}}
	\newcommand{\mj}{\left(\left[{(|\m|^2-k_1)^+}\right]_{\frac{k}{2}}+k_1\right)}
		\newcommand{\mn}{\left[|\m|^2\right]_{L}^n}
	\newcommand{\mjs}{\left[{(|\m|^2-k_1)^+}\right]_{\frac{k}{2}}}
	%\frac{k}{2^{j+1}}
	\newcommand{\eam}{e^{\alpha|\m|^2}}
	\newcommand{\vkp}{(v-k_{j+1})^+}
		\newcommand{\vkps}{\left[(v-k_{j+1})^+\right]^2}
\newcommand{\R}{\mathbb{R}}
\newcommand{\oyt}{\textup{osc}_{ B_r(y)}p}
\newcommand{\pyt}{p_{y,r}(t)}
\newcommand{\myt}{\m_{y,r}(t)}
\newcommand{\einf}{\textup{ess inf}}
\newcommand{\esup}{\textup{ess sup}}
\newcommand{\RN}{\mathbb{R}^N}
		\newcommand{\m}{{\mathbf{m} }}
			\newcommand{\n}{{\mathbf{n} }}
			\newcommand{\w}{{\bf w}}
	\newcommand{\pyr}{p_{y,r}(t)}
	\newcommand{\irn}{{\int_{\RN}}}
\newcommand{\ds}{\Delta^2}
\newcommand{\slp}{\mbox{div}\left(|\nabla p|^{s-2}\nabla p\right)}
\newcommand{\slpk}{\mbox{div}\left(|\nabla \pk|^{s-2}\nabla \pk\right)}
\newcommand{\wk}{\w_k}
\newcommand{\nwk}{\|\w_k\|_{\infty,\Omega}}
\newcommand{\psk}{p_k}
\newcommand{\npsk}{\|\nabla p_k\|_{q,\Omega}}
\newcommand{\npskt}{\|\nabla p_k\|_{2q,\Omega}}
\newcommand{\wko}{\w_{k-1}}
\newcommand{\nwko}{\|\w_{k-1}\|_{\infty,\Omega}}
\newcommand{\psko}{p_{k-1}}
\newcommand{\npsko}{\|\nabla p_{k-1}\|_{q,\Omega}}
\newcommand{\npskot}{\|\nabla p_{k-1}\|_{2q,\Omega}}
\newcommand*{\avint}{\mathop{\ooalign{$\int$\cr$-$}}}
\newcommand{\pk}{p_{y_k,r_k}(t)}
\newcommand{\mk}{m_{y_k,r_k}(t)}
\newcommand{\qdz}{Q_{\delta_0}(0)}
\newcommand{\bdz}{B_{\delta_0}(0)}
\newcommand{\id}{\int_{Q_{\delta_0}(0)}}
\newcommand{\iqo}{\int_{Q_{1}(0)}}
\newcommand{\ibo}{\int_{B_{1}(0)}}
\newcommand{\ibd}{\int_{B_{\delta_0}(0)}}
\newcommand{\io}{\int_{\Omega}}
\newcommand{\ios}{\int_{\Omega_s}}
\newcommand{\iot}{\int_{\Omega_{\tau}}}
\newcommand{\ioT}{\int_{\Omega_{T}}}
\newcommand{\iq}{\int_{Q_{r}(z)}}
\newcommand{\iqh}{\int_{Q_{\frac{r}{2}}(z)}}
\newcommand{\ib}{\int_{B_r(y)}}
\newcommand{\ibz}{\int_{B_1(0)}}
\newcommand{\ibh}{\int_{B_{\frac{r}{2}}(y)}}
\newcommand{\ibk}{\int_{B_{r_k}(y_k)}}
\newcommand{\iqk}{\int_{Q_{r_k}(z_k)}}
\newcommand{\mzk}{m_{z_k,r_k}}
\newcommand{\qyr}{Q_r(y,\tau) }
\newcommand{\qzr}{Q_r(z) }
\newcommand{\qyrk}{Q_{r_k}(y_k,\tau_k) }
\newcommand{\qzrk}{Q_{r_k}(z_k) }
\newcommand{\er}{E_r(y,\tau) }
\newcommand{\ezr}{E_r(z) }
\newcommand{\erk}{E_{r_k}(y_k,\tau_k) }
\newcommand{\ezk}{E_{r_k}(z_k) }
\newcommand{\byr}{B_r(y) }
\newcommand{\byrk}{B_{r_k}(y_k) }
\newcommand{\aiy}{\avint_{B_r(y)}  }
\newcommand{\air}{\avint_{B_\rho(y)}  }
\newcommand{\aizr}{\avint_{Q_\rho(z)}  }
\newcommand{\aiz}{\avint_{Q_r(z)}  }
\newcommand{\aiR}{\avint_{B_R(y)}  }
\newcommand{\aiyk}{\avint_{B_{r_k}(y_k)}  }
\newcommand{\aizk}{\avint_{Q_{r_k}(z_k)}  }
\newcommand{\ykr}{y_k+r_ky}
\newcommand{\tkr}{\tau_k+r_k^2\tau}
\newcommand{\maxt}{\max_{t\in[\tau-\frac{1}{2}r^2, \tau+\frac{1}{2}r^2]}}
\newcommand{\mat}{\max_{t\in[\tau-\frac{1}{2}r^2, \tau+\frac{1}{2}r^2]}}
\newcommand{\maxd}{\max_{\tau\in[-\frac{1}{2}\delta^2, \frac{1}{2}\delta^2]}}
\newcommand{\maxn}{\max_{\tau\in[-\frac{1}{2}, \frac{1}{2}]}}
\newcommand{\maxtk}{\max_{t\in[\tau_k-\frac{1}{2}r_k^2, \tau_k+\frac{1}{2}r_k^2]}}
\newcommand{\maxth}{\max_{t\in[\tau-\frac{1}{8}r^2, \tau+\frac{1}{8}r^2]}}
\newcommand{\mnp}{(\m\cdot\nabla p) }
\newcommand{\mnpl}{(\m\cdot\nabla (p-\ell_{j+1})^+) }
\newcommand{\nsk}{(n_k\cdot\nabla \sk) }
\newcommand{\ot}{\Omega_T }
\newcommand{\mrt}{\m_{y,r}(t) }
\newcommand{\mz}{\m_{z,r} }
\newcommand{\prt}{p_{y,r}(t) }
\newcommand{\pkt}{p_{y_k,r_k}(\tkr) }
\newcommand{\wks}{|w_{k}|^2 }
\newcommand{\sk}{\psi_{k} }
\newtheorem{theorem}{Theorem}[section]
\newtheorem{corollary}{Corollary}[section]
\newtheorem*{main}{Main Theorem}
\newtheorem{lemma}[theorem]{Lemma}
\newtheorem{clm}[theorem]{Claim}
\newtheorem{proposition}{Proposition}[section]
\newtheorem{conjecture}{Conjecture}
\newtheorem*{problem}{Problem}
\theoremstyle{definition}
\newtheorem{definition}[theorem]{Definition}
\newtheorem{remark}{Remark}
\newtheorem*{notation}{Notation}
\newcommand{\rn}{\mathbb{R}^N}
\newcommand{\gr}{G_\rho}
\newcommand{\pr}{\phi_\rho}
\newcommand{\br}{B_{\rho}}
\newcommand{\mr}{\m_{\rho}}
\newcommand{\thr}{\theta_{\rho}}
\newcommand{\f}{{\bf g}}
\newcommand{\eps}[1]{{#1}_{\varepsilon}}

%% Place the running title of the paper with 40 letters or less in []
 %% and the full title of the paper in { }.
\title[ a biological network formulation model
] %Use the shortened version of the full title
      {Partial regularity of weak solutions and life-span of smooth solutions to a biological network formulation model}

% Place all authors' names in [ ] shown as running head, Leave { } empty
% Please use `and' to connect the last two names if applicable
% Use FirstNameInitial.  MiddleNameInitial. LastName, or only last names of authors if there are too many authors
\author[Xiangsheng Xu]{}

% It is required to enter 2010 MSC.
\subjclass{Primary: 35A01, 35A09, 35M33, 35Q99.}
% Please provide minimum  5 keywords.
 \keywords{ Biological network formulation, cubic nonlinearity, life-span of smooth solutions, partial regularity of weak solutions.}

% Email address of each of all authors is required.
% You may list email addresses of all other authors, separately.
 \email{xxu@math.msstate.edu}
 %\email{email2@aimSciences.org}
 %\email{email3@ece.pdx.edu}

% Put your short thanks below. For long thanks/acknowlegements,
%please go to the last acknowlegments section.
%\thanks{The first author is supported by NSF grant xx-xxxx}

% Add corresponding author at the footnote of the first page if it is necessary. 
% Plase add $^*$ adjacent to the corresponding author's name on the first page. 
% The example shown in this template is if the first author is the corresponding author.
%\thanks{$^*$ Corresponding author: xxxx}

\maketitle

% Enter the first author's name and address:
\centerline{\scshape Xiangsheng Xu}
\medskip
{\footnotesize
% please put the address of the first author
 \centerline{Department of Mathematics \& Statistics}
   \centerline{Mississippi State University}
   \centerline{ Mississippi State, MS 39762, USA}
} % Do not forget to end the {\footnotesize by the sign }

\bigskip

% The name of the associate editor will be entered by an editorial staff
% "Communicated by the associate editor name" is not needed for special issue.
% \centerline{(Communicated by the associate editor name)}

	\begin{abstract}
In this paper we study partial regularity of weak solutions to the initial boundary value problem for the system $-\textup{div}\left[(I+{\bf m}\otimes {\bf m})\nabla p\right]=S(x),\ \ \partial_t{\bf m}-D^2\Delta {\bf m}-E^2({\bf m}\cdot\nabla p)\nabla p+|{\bf m}|^{2(\gamma-1)}{\bf m}=0$, where $S(x)$ is a given function and $D, E, \gamma$ are given numbers. This problem has been proposed as a PDE model for biological transportation networks. The mathematical difficulty is due to the fact that the system in the model features both a quadratic nonlinearity and a cubic nonlinearity. The regularity issue seems to have a connection to a conjecture by De Giorgi \cite{DE}. We also investigate the life-span of classical solutions. Our results show that local existence of a classical solution can always be obtained and the life-span of such a solution can be extended as far away as one wishes as long as the term $\|{\bf m}(x,0)\|_{\infty, \Omega}+\|S(x)\|_{\frac{2N}{3}, \Omega}$ is made suitably small, where $N$ is the space dimension and $\|\cdot\|_{q,\Omega}$ denotes the norm in $L^q(\Omega)$.

	\end{abstract}

\section{Introduction}Network formulation and transportation networks are fundamental processes in living systems \cite{AAFM}. 
%To give a few examples, 
%Typical example situations we have in mind are
%Example situations include
The angiogenesis of blood vessels, leaf venation, and creation of neural pathways in nervous systems are some of the well known examples. Tremendous interest has been shown for these phenomena from different scientific communities such as biologists, engineers, physicists, and computer scientists. Of particular interest is their property of optimal transport of fluids and other materials. The development of mathematical models for transportation networks and network formulation is a growing field. We would like to refer the reader to \cite{ABHMS} for a comprehensive review  and analysis of existing models.  

In this paper we are interested in the mathematical analysis of
%The second part of the proposal deals with 
a PDE model  first proposed by Hu and Cai  in \cite{HC}  that
%for biological network formulation. 
%The system
%, first proposed by  \cite{HC}, 
describes the pressure field of a network using a Darcy's type equation and the dynamics of the conductance network under pressure force effects. More precisely, 
let $\Omega$ be the network region, a bounded domain in $\mathbb{R}^N$, and $T$ a positive number. Set $\ot=\Omega\times(0,T)$. We study the behavior of solutions to the system
\begin{align}
-\textup{{div}}\left[(I+\m\otimes \m)\nabla p\right]&=S(x)\ \ \ \textup{in $\ot$,}\label{e1}\\
\partial_t\m-D^2\Delta \m-E^2\mnp\nabla p+|\m|^{2(\gamma-1)}\m&=0\ \ \ \textup{in $\ot$,}\label{e2}
\end{align}
coupled with the initial boundary conditions
\begin{align}
p(x,t)&=0, \ \ \ (x,t)\in \Sigma_T\equiv\partial\Omega\times(0,T),\label{e3}\\
\m(x,t)&=0, \ \ \ (x,t)\in \Sigma_T,\label{extm}\\
\m(x,0)&=\m_0(x),\ \ \ \ x\in\Omega\label{e4}
\end{align} 
for given function $S(x)$ and physical parameters $D, E, \gamma$ to be specified at a later time.
%This problem arises in the study of network formulation and transportation networks \cite{H, HC, AAFM}. Examples of such networks one has in mind are
%the angiogenesis of blood vessels, leaf venation, and creation of neural pathways in nervous systems. The development of mathematical models to describe them has attracted a lot of attention. Our problem here was first
%This system has been 
%proposed by Hu and Cai \cite{HC}. 
%to describe natural network formulation. 
%In this case 
Here the scalar function $p=p(x,t)$ is the pressure due to Darcy's law, while the vector-valued function $\m=(m_1(x,t), \cdots, m_N(x,t))^T$ is the conductance vector. The function $S(x)$ is the time-independent source term. Values of the parameters $D, E$, and $\gamma$ are
determined by the particular physical applications one has in mind. For example, in leaf venation we have $\gamma \in[\frac{1}{2},1]$\cite{HC}. 

In general nonlinear problems do not possess classical solutions. A suitable notion of a weak solution must be obtained for \eqref{e1}-\eqref{e4}. It turns out \cite{HMP} that we can introduce the following:

%	\noindent {\bf Definition.} 
\begin{definition}
	Let $\Omega, T$ be given as before. A pair $(\m, p)$ is said to be a weak solution to \eqref{e1}-\eqref{e4} in $\ot$ if:
	\begin{enumerate}
		\item[(D1)] $\m\in L^\infty\left(0,T; \left(W^{1,2}_0(\Omega)\cap L^{2\gamma}(\Omega)\right)^N\right),\  p\in L^\infty(0,T; W^{1,2}_0(\Omega)),\  \mnp \in L^\infty(0,T;  L^{2}(\Omega))$, $\partial_t\m\in L^2\left(0,T; \left(L^2(\Omega)\right)^N\right)$;
		\item[(D2)] $\m(x,0)=\m_0$ in $C\left([0,T]; \left(L^2(\Omega)\right)^N\right)$;
		\item[(D3)] Equations \eqref{e1} and \eqref{e2} are satisfied in the sense of distributions. That is, for a.e. $t\in (0,T)$ we have
		\begin{eqnarray*}
		\lefteqn{\io \nabla p\nabla\psi dx+\io \mnp(\m\cdot\nabla\psi)dx	=\io S(x)\psi dx}\\
	&&\mbox{for each $\psi\in W^{1,2}_0(\Omega)$ with $(\m\cdot\nabla\psi)\in L^2(\Omega)$},\\
	\lefteqn{	\io\partial_t m_i\psi dx+D^2\io\nabla m_i\nabla\psi dx+\io |\m|^{2(\gamma-1)}m_i\psi dx=E^2\io\mnp\partial_{x_i} p \psi dx}\\
	&& \mbox{ for each $\psi\in W^{1,2}_0(\Omega)\cap L^{2\gamma}(\Omega)$ such that $\partial_{x_i} p\psi\in L^2(\Omega)$, $i=1, \cdots, N$}.
		\end{eqnarray*}
	\end{enumerate}
\end{definition}
%A result 
\begin{lemma}[\cite{HMP}]\label{HMP}
	%Let $\Omega$ be a bounded domain in $\RN$ with Lipschitz boundary $\partial\Omega$. 
	Assume:
	\begin{enumerate}
		\item[\textup{(H1)}] $\Omega$ is a bounded domain in $\RN$ with Lipschitz boundary $\partial\Omega$;
		%with properties:
		%	\begin{enumerate}
		\item[\textup{(H2)}] $S(x)\in L^{2}(\Omega)$; 
		\item[\textup{(H3)}] $D, E\in (0, \infty), \gamma\in [1, \infty)$; and
		%	\end{enumerate}
		\item[\textup{(H4)}] $\m_0\in\left( W^{1,2}_0(\Omega)\cap L^{2\gamma}(\Omega)\right)^N$.
	\end{enumerate} 
	Then \eqref{e1}
	-\eqref{e4} has a weak solution. 
\end{lemma}
The proof in \cite{HMP} was based upon the formal gradient flow structure with respect to a suitable energy functional of the system, from which followed the  estimates
\begin{eqnarray}
\lefteqn{\frac{1}{2}\int_{\Omega}|\m(x,\tau)|^2dx+D^2\int_{\Omega_\tau}|\nabla \m|^2dxdt+E^2\int_{\Omega_\tau} \mnp^2dxdt}\nonumber\\
&&+\int_{\Omega_\tau}|\m|^{2\gamma}dxdt+2E^2	\int_{\Omega_\tau} |\nabla p|^2dxd\tau\nonumber\\
&=&\frac{1}{2}\int_{\Omega}|\m_0|^2dx+2E^2\int_{\Omega_\tau} S(x)pdxdt,\nonumber\\
%	\begin{eqnarray}
\lefteqn{\int_{\Omega_\tau}|\partial_t\m|^2dxdt+\frac{D^2}{2}\int_{\Omega}|\nabla \m(x,\tau)|^2dx+\frac{E^2}{2}\int_{\Omega}\mnp^2 dx}\nonumber\\
&&+\frac{E^2}{2}\int_{\Omega}|\nabla p|^2dx+\frac{1}{2\gamma}\int_{\Omega}|\m|^{2\gamma}dx\nonumber\\
&=&\frac{D^2}{2}\int_{\Omega}|\nabla \m_0|^2dx+\frac{E^2}{2}\int_{\Omega}(\m_0\cdot \nabla p_0)^2dx+\frac{1}{2\gamma}\int_{\Omega}|\m_0|^{2\gamma}dx\nonumber\\
&&+\frac{E^2}{2}\int_{\Omega}|\nabla p_0|^2dx,\label{f5}
\end{eqnarray}
where $\tau\in (0, T], \Omega_\tau=\Omega\times(0,\tau)$, and $p_0$ is the solution of the boundary value problem
\begin{eqnarray}
-\mbox{div}[(I+\m_0\otimes \m_0)\nabla p_0] &=& S(x),
\ \ \ \mbox{in $\Omega$,}\\
p_0&=& 0\ \ \ \mbox{on $\partial\Omega$.}
\end{eqnarray}

We refer the reader to \cite{AAFM,ABHMS,HMPS, L, X5, X7} for additional results concerning modeling, numerical simulations, and various properties of solutions.
%the corresponding stationary equations, and the one-dimensional problem. 
However, the general regularity theory remains fundamentally incomplete in high space dimensions. In particular, it is not known whether or not weak solutions develop singularities in space dimension $N\geq 3$. 
%The lonely exception seems to be 
%However, we do see from 
%\cite{X5,X7} where it is shown
%, the author shows 
%that a weak solution to the 2-dimensional stationary problem must be a classical one.

%from the proof of
%As we mentioned earlier, the existence of a weak solution was first established in \cite{HMP}. 
In \cite{LX}, Jian-Guo Liu and the author studied the partial regularity of weak solutions. In this context, we assume:
\begin{enumerate}
	%with properties:
	%	\begin{enumerate}
	\item[\textup{(A1)}] $S(x)\in L^{q}(\Omega)$ for some $q>\frac{N}{2}$; 
	\item[\textup{(A2)}] $D, E\in (0, \infty), \gamma\in (\frac{1}{2}, \infty)$; and
	%	\end{enumerate}
	\item[\textup{(A3)}] $\m_0\in\left( W^{1,2}_0(\Omega)\cap L^{\infty}(\Omega)\right)^N$.
\end{enumerate} 
%\end{eqnarray}
%It has been established in \cite{HMP} that , provided that, in addition to assuming $S(x)\in L^2(\Omega)$ and (H2), we also have
It is not difficult to see from our proof below that the conclusion of Lemma \ref{HMP} remains valid if we replace (H2)-(H4) in the lemma by the assumptions (A1)-(A3).
The authors in \cite{LX} considered the following quantities:
\begin{eqnarray}
\prt&=&\avint_{B_r(y)} p(x,t)dx=\frac{1}{|\byr|}\int_{B_r(y)} p(x,t)dx,\\
\m_{y, r}(t)&=&\avint_{B_r(y)}\m(x,t)dx,\\
\mz &=&\avint_{Q_r(z)}  \m(x,t)dxdt,\\
A_r(z)&=&\frac{1}{r^{N}}\mat\int_{B_r(y)}(p(x,t)-\prt)^2dx,\label{10m1}\\
%	D_r(z)&=&\mat|\mrt|,\\
%	C_r(z)&=&\mat\avint_{B_r(y)}|m(x,t)-\mrt|^2dx,\\
%The right choice for $E_r(z)$ seems to be 
%\begin{equation}
\ezr&=&\frac{1}{r^{N+2}}\int_{Q_r(z)}|\m-\mz|^2dxdt+A_r(z)+r^{2\beta},
%\end{equation}
\end{eqnarray}
where $\beta, r>0, z=(y, \tau)\in \ot$, $B_r(y)$ is the ball centered at $y$ with radius $r$, and $Q_r(z)$ is the cylinder $B_r(y)\times(\tau-\frac{1}{2}r^2,\tau+\frac{1}{2}r^2 )$. Here and in what follows it is understood that if $Q_r(z)$ (resp. $B_r(y)$) is not contained in $\ot$ (resp. $\Omega$) we replace $Q_r(z)$ by $Q_r(z)\cap \ot$ (resp. $B_r(y)\cap\Omega$).
A result of \cite{LX} asserts that $p\in C([0,T]; L^2(\Omega))$, and thus \eqref{10m1} makes sense.
The main result of \cite{LX} can be stated as follows:
\begin{lemma}[\cite{LX}] \label{LX}Let \textup{(H1)}, \textup{(A1)}-\textup{(A3)} be satisfied and $(\m, p)$ be a weak solution of \eqref{e1}-\eqref{e4}. Assume:
	\begin{enumerate}
		\item[\textup{(A4)}] $N=2$ or $3$.
	\end{enumerate}   If $z\in \ot$ is such that 
	\begin{equation}\label{lx1}
	\liminf_{r\rightarrow 0}\ezr =0\ \mbox{and $\limsup_{r\rightarrow 0}\mz<\infty$}, 
	\end{equation}
	then $z$ is a regular point. That is, there is a neighborhood of $z$ in which $\m$ is H\"{o}lder continuous. Furthermore, the set of all non-regular points, i.e., singular points, which we denote by $\mathbb{S}$, has parabolic Hausdorff dimension $N$.	
\end{lemma}
The proof in \cite{LX} is argument by contradiction. In the first part of this paper we shall investigate the partial regularity  issue from a different perspective. To introduce our results, we let
\begin{eqnarray}
%\lefteqn{
\mbox{osc}_{ B_r(y)}p
%}\nonumber\\
&=&\esup_{x_1, x_2\in B_r(y)}(p(x_1,t)-p(x_2,t))\nonumber\\
&&  \mbox{for $y\in \oo$ and a.e. $t\in (0,T)$, and}\\
\delta_r(y)&=&\esup_{0\leq t\leq T}	\mbox{osc}_{ B_r(y)}p.
\end{eqnarray}
To see that $\delta_r(y)$ is well-defined, we invoke Proposition 2.1 in \cite{LX} which states
%In view of , we have  that
\begin{equation}\label{pbd}
p\in L^\infty(\ot).
\end{equation}
%Hence,  .
%present a direct proof of a new partial regularity result, which can be viewed as a refinement of the preceding lemma. To be precise, we have: 
\begin{theorem}\label{linfty}Let \textup{(H1)}, \textup{(A1)}-\textup{(A3)} be satisfied and $(\m, p)$ be a weak solution of \eqref{e1}-\eqref{e4}.	If $y\in\oo$ is such that
	\begin{equation}\label{oscp1}
	\delta_r(y)\rightarrow 0\ \ \textup{as $r\rightarrow 0$,}
	\end{equation}
	then for each $n>0$ there is a $r>0$ with
	%suitably small we have
	\begin{eqnarray}
	|\m|^{2n}&\in& L^\infty(0,T;L^2(B_r(y)))\cap L^2(0,T; W^{1,2}(B_r(y))),\ \ \mbox{and}\\
	|\m|^{2n}\mnp^2&\in &L^1(\Lambda_r(y)),
	%\ \  \mbox{for each $n\geq 1$.} 
	\end{eqnarray}
	where
	%For $y\in\oo,\ r>0$ we let
	%consider the quantities
	\begin{eqnarray}
	%B_r(y)&=&B_r(y)\cap\Omega ,\\
	\Lambda_r(y)&= &B_r(y)\times(0,T).
	%\prt&=&\frac{1}{|B_r(y)|}\int_{B_r(y)}p(x,t)dx\equiv\avint_{B_r(y)}p(x,t)dx,\nonumber\\
	%\mrt&=&\avint_{B_r(y)}\m(x,t)dx,
	%\end{eqnarray}
	%where $B_r(y)$ denotes the ball centered at $y$ with radius $r$. 
	%Set
	%\begin{eqnarray}
	\end{eqnarray}
	%there holds $
\end{theorem}
If $N=2$, we obtain from \cite{X5} that 
\begin{equation}\label{osc1}
\left(\oyt\right)^2\ln\frac{R}{r}\leq c\int_{ B_R(y)}|\nabla p|^2dx+\int_{ B_R(y)}\mnp^2dx+cR^{2}
%\ \ \mbox{}.
%with $Q_R(z)\subset\ot$
\end{equation}
for a.e $t\in (0,T)$ and $0<r\leq R$. Here and in what follows the letter $c$ denotes a positive number whose value can always be computed from the given data at least in theory. Thus \eqref{oscp1} does hold.
In fact, this theorem is essentially Proposition 3.2 in \cite{X5}. To find conditions under which \eqref{oscp1} is true for $N>2$ turns out to be very challenging. The following theorem addresses this issue.
\begin{theorem}\label{partial}Let \textup{(H1)}, \textup{(A1)}-\textup{(A4)} all hold and $(\m, p)$ be a weak solution of \eqref{e1}-\eqref{e4}. If $y\in \Omega$ is such that
	\begin{eqnarray}
	\limsup_{r\rightarrow 0}\max_{0\leq t\leq T}|\m_{y, r}(t)|&<&\infty,\ \ \mbox{and}\label{con1}\\
	\limsup_{r\rightarrow 0}\esup_{0\leq t\leq T}\frac{1}{r^{N-2}}\int_{ B_r(y)}|\nabla \m(x,t)|^2dx&<&\infty.\label{con2}
	\end{eqnarray}
	then \eqref{oscp1} holds at $y$. Furthermore, the point $z=(y,\tau)$ satisfies \eqref{lx1} for each $\tau\in(0, T) $. That is, $\{y\}\times(0,T)$ are all regular points.
\end{theorem}
%We shall see that 
On account of (A4), \eqref{con1} and \eqref{con2} imply \eqref{oscp1} only when $N\leq 3$. 
%This seems to be the best result possible. 
The first conclusion in this theorem will be formulated as Theorem \ref{pro5} in Section \ref{sec3}.

Note that by its definition the set of regular points is always open. 
%Denote by $\mathbb{S}$ the set of all singular points. 
We have not been able to obtain the Hausdorff measure of $\mathbb{S}$ in the context of this theorem. However, if $N=2$, then \eqref{con2} is satisfied for all $y\in\Omega$. As for \eqref{con1} in this case, we can infer from the argument given in (\cite{G}, p. 104) that
%If $z=(y,\tau)\in H_\varepsilon$, then we 
%	For each $y\in \Omega$,	we have
\begin{equation}
\max_{t\in[0, T]}	|\m_{y,r}(t)|<\max_{t\in[0, T]}|\m_{y,R}(t)|+c\ln\frac{R}{r}\ \ \textup{for all $0<r\leq R$}. 
\end{equation}
That is, for each $\varepsilon>0$ we have
\begin{equation}
\lim_{r\rightarrow 0}r^\varepsilon \max_{t\in[0, T]}	|\m_{y,r}(t)|=0.
\end{equation}
%In view of Theorem 3 in (\cite{EG}, p.77) and Theorem 2.1 in (\cite{G}, p.100) our partial regularity result seems to suggest that at each time level $t=t_0\in (0,T)$ the singular set $\mathbb{S}$ is at most countable. In this regard, 
A recent result of the author \cite{X7} indicates that
%We would like to mention that it has been conjectured \cite{X5} that
 $\mathbb{S}$ is empty when $N=2$ and some additional assumptions on $\partial \Omega$ and the given data are satisfied.
 %, and this conjecture is still open.%provided that $N=2$. 
%If $N=3$, the set $S\cap \{t=t_0\}$ has Hausdorff dimension 1.  
% Recently, Jian-Guo Liu and the author \cite{LX} obtained a partial regularity theorem for \eqref{e1}-\eqref{e4}. It states that the parabolic Hausdorff dimension of the set of singular points can not exceed $N$, provided that $N\leq 3$.

There are two very interesting mathematical features associated with the system. The first one concerns the elliptic coefficient matrix $A\equiv I+\m\otimes\m$ in the first equation.
Remember that the existing regularity theory for elliptic equations requires that the
largest eigenvalue $\lambda_l$ of $A$ and the smallest one $\lambda_s$  be
suitably ``balanced''. A typical example of such assumptions is that
%\begin{equation*}
$\lambda_l\leq c\lambda_s$
%\end{equation*}
and $\lambda_s$ is an $A_2$-weight \cite{HKM}. That is, we have
%$\lambda_s$ satisfies the inequality
\begin{equation}
\aiy \lambda_s dx\aiy\frac{1}{\lambda_s}dx\leq c\ \ \ \mbox{for each ball $B_r(y)\subset\Omega$.}
\end{equation}
%In our case,  we have
The matrix $A$ here satisfies
\begin{equation}\label{ellip}
|\xi|^2\leq (A\xi\cdot\xi)\leq (1+|\m|^2)|\xi|^2\ \ \mbox{for each $\xi\in \mathbb{R}^N$.}
\end{equation}Thus if $\m$ is not locally bounded a priori, our case lies outside 
the scope 
of the standard elliptic regularity theory. Our situation seems to be related to a conjecture by De Giorgi \cite{DE} (also see \cite{OZ}), which, in our context, roughly says  that 
\begin{equation}
\esup_{0\leq t\leq T}\io\exp\sqrt{1+|\m|^2}\ dx<\infty\ \ \mbox{impies}\ \ p\in L^\infty(0,T;C_{\textup{loc}}(\Omega)).
\end{equation}
This is indeed true if the space dimension is 2 \cite{X5}. Unfortunately,  the membership of $p$ in $L^\infty(0,T;C_{\textup{loc}}(\Omega))$ is not enough to bridge the gap to the local boundedness of $\m$. As we shall see in Section \ref{sec3}, we need to strengthen the assumption to $p\in L^\infty(0,T;C_{\textup{loc}}^\alpha(\Omega))$ for some $\alpha\in (0,1)$ in order to show that $\m$ is locally bounded.
%lead to the conclusion that $\m$ is locally bounded. We will show in  that the conclusion is valid if we 
%The quadratic nonlinearity in $\m$ also poses a mathematical challenge. 
The second one is the tri-linear term $\mnp\nabla p$ in the system, which actually represents a cubic nonlinearity. Currently, there has not been much research work done on this type of nonlinearities.

In the second part of this paper we study the existence of a weak solution that possesses the additional property
\begin{enumerate}
	\item[(D4)] $\|\m\|_{\infty,\ot}<\infty$ and $\sup_{0\leq t\leq T}\|\nabla p\|_{q,\Omega}<\infty$ for each $q>1$.
\end{enumerate}
We would like to remark that if $N=2$ then the two conditions in (D4) are equivalent (see Lemma \ref{mpe} below).
\begin{theorem}\label{prop1} Let \textup{(A2)} hold.
	If $\partial\Omega$ is $C^{2,\alpha}$, $S(x)\in C^\alpha(\overline{\Omega}) $, and $\m_0\in \left(C^{2,\alpha}(\overline{\Omega})\right)^N$ for some $\alpha\in (0,1)$,
	%and $$ is H\"{o}ler continuous on $$,		 
	then a weak solution to \eqref{e1}-\eqref{e4} with the additional property \textup{(D4)} is also a classical one.
\end{theorem}
The proof of this proposition will be presented at the end of Section 2.
%A result in \cite{X5} asserts that
\begin{theorem}\label{locexis}
	Let \textup{(A1)-(A3)} hold.	Assume:
	\begin{enumerate}
		\item[\textup{(H5)}] $\m_0$ is H\"{o}lder continuous on $\overline{\Omega}$;
		\item[\textup{(H6)}]$\partial\Omega$ is $C^{1}$.
	\end{enumerate} 
	Then
	there is a positive number $T$ determined by the given data such that \eqref{e1}-\eqref{e4} has a weak solution $(\m, p)$ with the property \textup{(D4)} on $\ot$.
\end{theorem}
The next theorem reveals how the life-span of a classical solution
depends on the size of given data.
\begin{theorem}\label{larexis} Let the assumptions of Theorem \ref{locexis} be satisfied. For each $T>0$ there is a positive number $\delta=\delta(T)$ such that \eqref{e1}-\eqref{e4} has
	a weak solution on $\ot$ with the property (D4) whenever $\|S(x)\|_{\frac{2N}{3},\Omega}+\|\m_0\|_{\infty, \Omega}\leq \delta$.
\end{theorem}
We believe that the fact that the number $\delta$ in the theorem has
to depend on $T$ is related to the time-independence of the source term $S(x)$. If $S$ is not identically $0$, then we always have $\|S\|_{q,\Omega\times[0,\infty)}=\infty$ for any $q>1$. 
%  Also see \cite{L}.
We speculate that if the source term $S$ is a function of both time and space and $\|S\|_{q,\Omega\times[0,\infty)}$ is suitably small for some $q>1$ we may be able to prove the existence of a classical solution on $\Omega\times[0,\infty)$ \cite{L}. However, we must point out that the time-dependence of $S$ will cause \eqref{f5} to fail, and thus a new existence theorem other than the one in \cite{HMP} is needed.
%\begin{theorem}\label{main} 
%	If $N=2, S(x)\in L^2(\Omega), |\m_0|\in L^\infty(\Omega) $, for each $T>0$ there is a classical solution to \eqref{e1}-\eqref{e4}.
%\end{theorem}

%We remark that the assumption that $S(x)\in L^2(\Omega)$ in this theorem can be weakened to $S(x)\in L^{\frac{4}{3}}(\Omega)$.

Nonlinearities in partial differential equations often play a rather peculiar
role in blow-up of solutions. In this connection we would like to mention the well known Fujita phenomenon. It roughly says that for certain types of nonlinearities solutions exists globally for some data, while for some other data solutions blow up no matter how small or smooth these data are \cite{GP}. Note that Theorem \ref{larexis} is neither a global existence result nor a blow-up result. As we mentioned earlier, many regularity problems associated with \eqref{e1}-\eqref{e4} remain open.
%show
%that the latter cannot happen in our case in spite of the new mathematical features involved.
%Even though our results are not surprising, they should be viewed in terms of the  mentioned earlier.

The rest of the paper is organized as follows: In section \ref{sec2} we collect some preparatory lemmas. Here we take or refine some relevant classical results. In Section \ref{sec3} we investigate regularity and partial regularity of weak solutions. We show that $p\in L^\infty(0,T; C^\alpha(\overline{\Omega}))$ leads to H\"{o}lder continuity of $\m$. The proof of Theorem \ref{partial} is also given here. Section \ref{sec4} is devoted to the proof of Theorems \ref{locexis} and \ref{larexis}. A successive approximation scheme is employed for the second theorem. The mathematical challenge here is that one must show that the entire approximate sequence converges in a suitable sense. 
%In the last section we consider the local boundedness estimates for $p$. This is motivated by the fact that in elliptic theory local boundedness estimates often result in local H\"{o}lder continuity. At least in the case where $N=2$ we 
%can infer from an argument in \cite{Y} that 
%However, our results here only represent a modest first step.
\section{Preliminaries}\label{sec2}
In this section we prepare some background results. Some of them are well-known and some of them are a  refinement of known results so that they fit our purpose.

Our first result is an elementary inequality whose proof is contained in (\cite{O}, p. 146-148). 
\begin{lemma}\label{plap}Let $x,y$ be any two vectors in $\RN$. Then:
	\begin{enumerate}
		\item[\textup{(i)}] For $\gamma\geq 1$,
		\begin{equation*}
		\left(\left(|x|^{2\gamma-2}x-|y|^{2\gamma-2}y\right)\cdot(x-y)\right)\geq \frac{1}{2^{2\gamma-1}}|x-y|^{2\gamma};
		\end{equation*}
		\item[\textup{(ii)}] For $\frac{1}{2}<\gamma\leq 1$,
		\begin{equation*}
		\left(|x|+|y|\right)^{2-2\gamma}\left(\left(|x|^{2\gamma-2}x-|y|^{2\gamma-2}y\right)\cdot(x-y)\right)\geq (2\gamma-1)|x-y|^2.
		\end{equation*}
	\end{enumerate}
\end{lemma}
For each $q\geq 1$ we define the Banach space $L^*_{q}(\Omega)$, where $\Omega\subseteq\RN$, by
$$L^*_{q}(\Omega)=\{f: \mbox{there is a number $c\geq 0$ such that $|\{x\in\Omega: |f(x)|\geq t\}|\leq\frac{c^q}{t^q}$ for all $t>0$} \}.
$$
The smallest $c$ such that the above inequality holds is the norm of $f$ in $L^*_{q}(\Omega)$. We easily see
\begin{eqnarray}
\|f\|_{L^*_{q}(\Omega)}&\leq &\|f\|_{L^q(\Omega)}.
\end{eqnarray}
%If $\Omega$ is bounded, then
Moreover, for each measurable subset $\Omega\subset\RN$ and each $\varepsilon\in (0, q-1]$ we have from \cite{BBC} that
\begin{eqnarray}
\int_{\Omega}|f|^{q-\varepsilon}dx&\leq &\frac{q}{\varepsilon} |\Omega|^{\frac{\varepsilon}{q}}\left(\|f\|_{L^*_{q}(\Omega)}\right)^{q-\varepsilon}. \label{wlp}
\end{eqnarray}

The next two lemmas deal with sequences of nonnegative numbers
which satisfy certain recursive inequalities.
\begin{lemma}\label{ynb}
	Let $\{y_n\}, n=0,1,2,\cdots$, be a sequence of positive numbers satisfying the recursive inequalities
	\begin{equation*}
	y_{n+1}\leq cb^ny_n^{1+\alpha}\ \ \textup{for some $b>1, c, \alpha\in (0,\infty)$.}
	\end{equation*}
	If
	\begin{equation*}
	y_0\leq c^{-\frac{1}{\alpha}}b^{-\frac{1}{\alpha^2}},
	\end{equation*}
	then $\lim_{n\rightarrow\infty}y_n=0$.
\end{lemma}
This lemma is well-known. See, e.g., (\cite{D}, p.12). Here we give a brief proof. We can easily show
\begin{eqnarray}
y_{n+1}&\leq& c^{\frac{(1+\alpha)^{n+1}-1}{\alpha}}b^{\frac{(1+\alpha)^{n+1}-1}{\alpha^2}-\frac{n+1}{\alpha}}y_0^{(1+\alpha)^{n+1}}\nonumber\\
&=&c^{-\frac{1}{\alpha}}b^{-\frac{1}{\alpha^2}}\left(c^{\frac{1}{\alpha}}b^{\frac{1}{\alpha^2}}y_0\right)^{(1+\alpha)^{n+1}}b^{-\frac{n+1}{\alpha}}.
\end{eqnarray}
Therefore, if $b>1$ and $c^{\frac{1}{\alpha}}b^{\frac{1}{\alpha^2}}y_0\leq 1$, then we have that $\lim_{n\rightarrow\infty}y_n=0$.
%Our  global existence  assertion in the case of small initial %data relies on the following lemma.

\begin{lemma}\label{small}
	Let $\alpha,\lambda\in (0,\infty)$ be given and $\{b_k\}$ a sequence of nonnegative numbers with the property
	\begin{equation*}
	b_k\leq b_0+\lambda b_{k-1}^{1+\alpha}\ \ \textup{for $k=1,2,\cdots.$}
	\end{equation*}
	If $2\lambda(2b_0)^\alpha<1$, then
	\begin{equation*}
	b_k\leq \frac{b_0}{1-\lambda(2b_0)^\alpha}\leq 2b_0\ \ \textup{for all $k\geq 0$.}
	\end{equation*}
\end{lemma}
This lemma can easily be established via induction.
\begin{lemma}\label{hcon}
	Let $\Omega$ be a bounded domain in $\RN$ with Lipschitz boundary $\partial\Omega$. Assume that $w$ is a weak solution of the initial boundary value problem
	\begin{eqnarray*}
		\partial_tw-D^2\Delta w&=&\textup{div}\f+ g_0\ \ \textup{in $\ot$,}\\ 
		w&=&0\ \ \textup{on $\Sigma_T$,}\\
		w(x,0)&=& w_0(x),
	\end{eqnarray*}
	where $w_0$ is H\"{o}lder continuous on $\overline{\Omega}$ and
	$|\f|^2, g_0\in L^q(\ot)$ for some $q>1+\frac{N}{2}$. Then $w$ is H\"{o}lder continuous on $\overline{\ot}$. That is, there is
	a number $\beta\in (0, 1)$ such that
	\begin{equation*}
	\sup_{(x_1,t_1),\  (x_2,t_2)\in \ot}\frac{|w(x_1,t_1)-w(x_2,t_2)|}{\left(|x_1-x_2|+|t_1-t_2|^{\frac{1}{2}}\right)^\beta}<\infty.
	\end{equation*}
\end{lemma}
This result is well-known, and it can be found, for example, in \cite{LSU}. Next, we cite a result from (\cite{R}, p.82).
\begin{lemma}\label{lqe}Let (H6) hold and assume
	\begin{enumerate}
		\item[\textup{(L1)}] $A(x)$ is an $N\times N$ matrix whose entries are
		continuous functions on $\overline{\Omega}$, satisfying the uniform
		ellipticity condition
		\begin{equation*}
		\lambda|\xi|^2\leq (A(x)\xi\cdot\xi)\leq\frac{1}{\lambda}\ \ \textup{on $ \Omega$ for some $\lambda>0$ .}
		%\item[(L2)] description
		\end{equation*}
	\end{enumerate}
	%Let 
	If $u$ is a weak solution to 
	%Consider 
	the boundary value problem
	\begin{eqnarray*}
		-\textup{div}\left(A\nabla u\right)&=&\textup{div}\f+g_0\ \ \textup{in $\Omega$,}\\
		u&=&0\ \ \textup{on $\Omega$,}
	\end{eqnarray*}
	then for each $q>1$ there is a positive $c=c(N, q, \Omega)$ with the property
	\begin{equation*}
	\|\nabla u\|_{q,\Omega}\leq c\left(\|\f\|_{q,\Omega}+\|g_0\|_{\frac{Nq}{N+q},\Omega}\right).
	\end{equation*}
\end{lemma}
We can easily infer from the preceding two lemmas that (D4) can be replaced by
\begin{enumerate}
	\item[(D4)$^\prime$] $\|\m\|_{\infty,\ot}<\infty$ and there is a $q>1+\frac{N}{2}$ such that $\sup_{0\leq t\leq T}\|\nabla p\|_{2q,\Omega}<\infty$.
\end{enumerate}
\begin{lemma}\label{lwb} Let $\w$ be a weak solution of
	%We consider 
	the initial boundary value problem
	\begin{eqnarray}
	\partial_t\w-D^2\Delta\w+|\w|^{2(\gamma-1)}\w&=&\f\ \ \textup{in $\ot$,}\label{wb1}\\
	\w&=&0\ \ \textup{on $\Sigma_T$,}\\
	\w(x,0)&=&\m_0(x)\ \ \textup{on $\Omega$.}
	\end{eqnarray}
	Then 
	%for each $q>1+\frac{N}{2}$ 
	there is a positive number $c=c(N)$ such that
	\begin{equation*}
	\|\w\|_{\infty,\ot}\leq c\left(\|\m_0\|_{\infty,\Omega}+|\ot|^{\frac{1}{N+2}}\sup_{0\leq t\leq T}\|\f\|_{N,\Omega}\right).
	\end{equation*}
\end{lemma}
\begin{proof} Even though \eqref{wb1} is a system, the classical method due to De Giorgi is still applicable. 
	%This result is well known. 
	Here we give an outline of the proof.
	%, using .
	%Denote by $w_i$ (resp. $(\f\cdot\w)$) the i-th component of
	%$\w$ (resp. $\f$). Then we have 
	%\begin{equation}\label{wb1}
	%	\partial_t w_i-D^2\Delta w_i+|\w|^{2(\gamma-1)} w_i=(\f\cdot\w)\ \ \textup{in $\ot$.}
	%	\end{equation}
	%	Without loss of generality, we assume that
	%	\begin{equation*}
	%	\sup_{\ot} w_i=\sup_{\ot} |w_i|.
	%	\end{equation*}
	Set
	\begin{equation}\label{initial}
	M=\||\m_0|^2\|_{\infty, \Omega}.
	\end{equation}
	Then define
	\begin{equation*}
	k_n=k-\frac{k}{2^n}+M, \ \ \ n=0,1,2,\cdots,
	\end{equation*}
	where $k>0$ is a number to be determined. Let
	\begin{equation*}
	A_n(t)=\{x\in\Omega: |\w|^2(x,t)\geq k_n \}.
	\end{equation*} 
	Without loss of generality,  assume $N>2$. Use $(|\w|^2-k_{n})^+\w$ as a test function in \eqref{wb1} to derive, with the aid of the Gagliardo-Nirenberg-Sobolev inequality, that
	\begin{eqnarray}
	\lefteqn{\frac{1}{4}\frac{d}{dt}\io\left[(|\w|^2-k_{n})^+\right]^2dx+D^2\io(|\w|^2-k_{n})^+\left|\nabla\w\right|^2dx}\nonumber\\
	&&+\frac{D^2}{2}\io |\nabla(|\w|^2-k_{n})^+|^2dx\nonumber\\
	&\leq& \io (\f\cdot\w)(|\w|^2-k_{n})^+dx\nonumber\\
	&\leq&\|(|\w|^2-k_{n})^+\|_{\frac{2N}{N-2}}\|(\f\cdot\w)\chi_{A_{n}(t)}\|_{\frac{2N}{N+2}}\nonumber\\
	&\leq&c\|\nabla(|\w|^2-k_{n})^+\|_{2}\|(\f\cdot\w)\chi_{A_{n}(t)}\|_{\frac{2N}{N+2}}\nonumber\\
	&\leq&	\frac{D^2}{4}\io\left|\nabla(|\w|^2-k_{n})^+\right|^2dx+c\|(\f\cdot\w)\|_{N,\Omega}^2|A_{n}(t)|,
	\end{eqnarray}
	from whence follows
	\begin{eqnarray}
	\lefteqn{\sup_{0\leq t\leq T}\io\left[(|\w|^2-k_{n})^+\right]^2dx+\ioT\left|\nabla(|\w|^2-k_{n})^+\right|^2dxd\tau}\nonumber\\
	&\leq &c\sup_{0\leq t\leq T}\|(\f\cdot\w)\|_{N,\Omega}^2\int_{0}^{T}|A_{n}(t)|dt.
	\end{eqnarray}
	Use the Gagliardo-Nirenberg-Sobolev inequality again to obtain
	\begin{eqnarray}
	\lefteqn{\ioT\left[(|\w|^2-k_{n})^+\right]^{2\frac{N+2}{N}}dxdt}\nonumber\\
	&=&\int_{0}^{T}\io\left[(|\w|^2-k_{n})^+\right]^2\left[(|\w|^2-k_{n})^+\right]^{\frac{4}{N}}dxdt\nonumber\\
	&\leq&\int_{0}^{T}\left(\io\left[(|\w|^2-k_{n})^+\right]^{\frac{2N}{N-2}}\right)^{\frac{N-2}{N}}\left(\io\left[(|\w|^2-k_{n})^+\right]^{2}\right)^{\frac{2}{N}}dt\nonumber\\
	&\leq& c\int_{0}^{T}\io\left|\nabla(|\w|^2-k_{n})^+\right|^2dxdt\left(\sup_{0\leq t\leq T}\io\left[(|\w|^2-k_{n})^+\right]^{2}\right)^{\frac{2}{N}}\nonumber\\
	&\leq&c\sup_{0\leq t\leq T}\|(\f\cdot\w)\|_{N,\Omega}^{2+\frac{4}{N}}\left(\int_{0}^{T}|A_{n}(t)|dt\right)^{1+\frac{2}{N}}.\label{wb11}
	%\left(\ioT (\f\cdot\w)(|\w|^2-k_{n})^+dxdt\right)^{1+\frac{2}{N}}\nonumber\\
	%	&\leq &c\left(\int_{\{|\w|^2\geq k_{n}\}\cap\ot}|(\f\cdot\w)|^{\frac{2(N+2)}{N+4}} dxdt\right)^{\frac{N+4}{2N}}\nonumber\\
	%	&&\cdot\left(\ioT\left[(|\w|^2-k_{n})^+\right]^{2\frac{N+2}{N}}dxdt\right)^{\frac{1}{2}},
	\end{eqnarray}
	%	from whence follows
	Set
	\begin{equation*}
	y_n=\int_{0}^{T}|A_{n}(t)|dt=|\{(x,t)\in \ot: |\w|^2(x,t)\geq k_n\}|.
	\end{equation*}
	We can easily show that 
	\begin{eqnarray}
	y_{n+1}=\int_{0}^{T}|A_{n+1}(t)|dt&\leq& \frac{4^{\frac{N+2}{N}(n+1)}}{k^{\frac{2(N+2)}{N}}}\ioT \left[(|\w|^2-k_{n})^+\right]^{2\frac{N+2}{N}}dxdt\nonumber\\
	&\leq& \frac{c4^{\frac{N+2}{N}(n+1)}\sup_{0\leq t\leq T}\|(\f\cdot\w)\|_{N,\Omega}^{2+\frac{4}{N}}}{k^{\frac{2(N+2)}{N}}}y_n^{1+\frac{2}{N}}.
	\end{eqnarray}
	This puts us in a position to apply Lemma \ref{ynb}. Upon doing so, we arrive at
	\begin{equation}\label{9w1}
	\||\w|^2\|_{\infty, \Omega}\leq k+M,
	\end{equation}
	provided that 
	\begin{eqnarray}
	k&=&cy_0^{\frac{1}{N+2}}\sup_{0\leq t\leq T}\|(\f\cdot\w)\|_{N,\Omega}\nonumber\\
	&\leq &cy_0^{\frac{1}{N+2}}\||\w|\|_{\infty, \Omega}\sup_{0\leq t\leq T}\|\f\|_{N,\Omega}\nonumber\\
	&\leq &\varepsilon 	\||\w|^2\|_{\infty, \Omega}+c(\varepsilon)y_0^{\frac{2}{N+2}}\sup_{0\leq t\leq T}\|\f\|_{N,\Omega}^2.
	%,\ \ \ c=c(N).
	\end{eqnarray}
	Use this in \eqref{9w1} to yield the desired result.
	%	Here we have made use of the fact 
	%	that $\int_{0}^{T}|A_{n}(t)|dt=|\{(x,t)\in \ot: |\w|^2(x,t)\geq k_n\}|$. 
	%Trivially, we have $y_0\leq |\ot|$.
	This completes the proof.
\end{proof}
\begin{lemma}\label{mpe}
	Let (H2), (H3), (H5), and (H6) hold and $(\m, p)$ be a weak solution of \eqref{e1}-\eqref{e4}. Assume $N=2$. Then $\|\m\|_{\infty,\ot}<\infty$ if and only if $\sup_{0\leq t\leq T}\|\nabla p\|_{q,\Omega}<\infty$ for each $q>1$.
\end{lemma}
\begin{proof}
	Suppose that $\|\m\|_{\infty,\ot}<\infty$. Then Equation \eqref{e1}
	is uniformly elliptic. A result in \cite{M} asserts that there is a $q>2$ such that
	\begin{equation}
	\|\nabla p\|_{q,\Omega}\leq c\|S(x)\|_{\frac{2q}{2+q},\Omega}.
	\end{equation}
	This together with an argument in \cite{Y} [also see (\cite{LSU}, p.182)] implies that $\m$ is H\"{o}lder continuous on $\overline{\ot}$. Thus Lemma \ref{lqe} becomes applicable to \eqref{e1}. This yields the desired result. 
	%Recall from (D1)
	%	that $\m\in L^\infty\left(0,T;\left(W^{1,2}(\Omega)\right)^N\right)$. This together with the assumption $N=2$ implies that $\m\in L^\infty\left(0,T;\left(\textup{VMO}(\RN)\right)^N\right)$, where $\m$ is understood to be $0$ outside $\Omega$. Thus Theorem 2.1 in \cite{DI} becomes applicable to \eqref{e1}. Using it yields the desired result.
	
	Now assume that $\sup_{0\leq t\leq T}\|\nabla p\|_{q,\Omega}<\infty$ for each $q>1$. Fix $\tau\in (0, T]$. By Lemma \ref{lwb}, there is a positive number $c=c(N)$ such that
	\begin{eqnarray}
	\|\m\|_{\infty,\Omega\times(0,\tau)}&\leq& c\left(\|\m_0\|_{\infty,\Omega}+|\Omega\times(0,\tau)|^{\frac{1}{N+2}}\sup_{0\leq t\leq \tau}\|\mnp\nabla p\|_{N,\Omega}\right)\nonumber\\
	&\leq&c+c\tau^{\frac{1}{N+2}}\|\m\|_{\infty,\Omega\times(0,\tau)}\sup_{0\leq t\leq T}\||\nabla p|^2\|_{N,\Omega}\nonumber\\
	&\leq &c+c\tau^{\frac{1}{N+2}}\|\m\|_{\infty,\Omega\times(0,\tau)},\label{n22}
	\end{eqnarray}
	where $c$ is independent of $\tau$. Hence we can choose $\tau$ so that 
	\begin{equation}
	\mbox{the coefficient of $\|\m\|_{\infty,\Omega\times(0,\tau)}$ on the right-hand side of \eqref{n22} $\equiv c\tau^{\frac{1}{N+2}}<1$}.
	\end{equation}
	This immediately gives $\|\m\|_{\infty,\Omega\times(0,\tau)}<\infty$. Obviously, $[0,T]$ can be divided into a finite number of subintervals with
	each one of them having length less than $\tau$. Apply the preceding argument successively to each one of the subintervals, starting with $[0,\tau]$. The desired result follows.
\end{proof}
Before we conclude this section, we offer the proof of Theorem \ref{prop1}.
\begin{proof}[Proof of Theorem \ref{prop1}]We will only give an outline of the proof, leaving many well-known technical details out.  Assume (D4) and (H3). By the Calderon-Zygmund inequality for parabolic equations \cite{LSU}, we have
	\begin{equation*}
	\partial_t\m, \ \Delta\m \in \left(L^q(\ot)\right)^N\ \ \textup{for each $q>1$.}
	\end{equation*}  Differentiate
	both sides of \eqref{e2} with respect to each one of the space variables and apply Lemma \ref{hcon} to the resulting
	equations in a suitable way to conclude that $\nabla\m$ is H\"{o}lder continuous on $\overline{\ot}$. As a result, the classical Schauder estimates (\cite{GT}, p.107) become applicable to \eqref{e1}. Upon applying, we yield that $p\in L^\infty(0,T; C^{2,\alpha}(\overline{\Omega}))$ for some $\alpha\in (0,1)$. 
	Differentiate \eqref{e1} with respect to $t$ to obtain
	\begin{equation}\label{pte}
	-\textup{div}\left[(I+\m\otimes \m)\nabla \partial_tp\right]=\textup{div}\left(\partial_t\m\otimes \m\nabla p\right)
	+\textup{div}\left(\m\otimes\partial_t \m\nabla p\right)\ \ \ \textup{in $\ot$}.
	\end{equation}
	This puts us in a position to use Lemma \ref{lqe}, from which follows 
	\begin{equation*}
	\nabla \partial_tp\in  \left(L^q(\ot)\right)^N\ \ \textup{for each $q>1$.}
	\end{equation*}
	Differentiate both sides of \eqref{e2} with respect to $t$ and apply the Calderon-Zygmund inequality to the resulting equation to obtain
	\begin{equation*}
	\frac{\partial^2 \m}{\partial t^2}, \ \Delta\partial_t\m \in \left(L^q(\ot)\right)^N\ \ \textup{for each $q>1$.}
	\end{equation*} 
	Now the right-hand side of \eqref{pte} is H\"{o}lder continuous in the space variables, and hence we can apply the Schauder estimates to it to get 
	%implies
	%	\begin{equ}
	%This puts us in a posAstion 
	\begin{equation*}
	\partial_tp\in L^q(0,T; C^{2,\alpha}(\overline{\Omega}))\ \ \textup{for each $q>1$.}
	\end{equation*}
	This implies that $\nabla p$ is H\"{o}lder continuous on $\overline{\ot}$. On the other hand, owing to (H3), the term
	$|\m|^{2(\gamma-1)}\m$ is also H\"{o}lder continuous. We can conclude
	from the parabolic Schauder estimates \cite{K} that $\m$ is a classical solution of \eqref{e2}.
\end{proof}
\section{Partial regularity of weak solutions}\label{sec3}
We begin this section by proving Theorem \ref{linfty}. To this end, we introduce the following notation. For  $-\infty<\ell<L<\infty$ denote by $\theta_{\ell, L}(s)$ the function
\begin{equation}
\theta_{\ell, L}(s)=\left\{\begin{array}{ll}
L&\mbox{if $s\geq L$,}\\
s&\mbox{if $\ell<s< L$,}\\
\ell&\mbox{if $s\leq \ell$.}
\end{array}\right.
\end{equation}
%\section{Proof of Theorem \ref{main}}
%We will assume that $N=2$ throughout this section. Let $(\m, p)$ be a weak solution to \eqref{e1}-\eqref{e4}. Denote by $R$ the distance between $y$ and $\partial\Omega$. Then 

\begin{proof}[Proof of Theorem \ref{linfty}] We only need to consider the case where $N>2$.
	Let $y$ be given as in the theorem. First we assume that $y$ is an interior point. We will show that for each $n>0$ there is a $r\in (0,\dy)$ such that
	\begin{equation}\label{10ow1}
	\mbox{$|\m|^{2(1+n)(1+\frac{2}{N})}\in L^1(\Lambda_{\frac{r}{2}}(y))$ whenever $|\m|^{2(1+n)}\in L^1(\Lambda_r(y))$.}
	\end{equation} 
	By iterating this result, we obtain our theorem.
	
	To see \eqref{10ow1}, for $r\in (0,\dy)$
	selected as below we pick a $C^\infty $ cut-off function with the properties
	% we define a sequence $\{r\}_{j=0}^{\infty}$ by
	%	Then set
	%	\begin{eqnarray}
	%	r_j&=&\frac{r}{2}+\frac{r}{2^{j+1}}.
	%\\
	%	k_j&=&M+k-\frac{k}{2^j}\ \mbox{and}\\
	%	\ell_j&=&p_{y,r}(t)+\ell-\frac{\ell}{2^j},\ \ j=0, 1,2,\cdots,
	%	\end{eqnarray}
	%	where $k, \ell$ are two positive numbers to be determined and $M$ is given as in \eqref{initial}.
	%	Then choose a sequence of smooth functions $\eta$ so that
	\begin{eqnarray}
	\eta&=&1\ \ \ \textup{on $B_{\frac{r}{2}}(y)$,}\label{cutoff1}\\
	\eta&=&0\ \ \ \textup{outside $B_{r}(y)$,}\\
	0&\leq&\eta\leq 1\ \ \mbox{on $\rn$},\\
	|\nabla\eta|&\leq&\frac{c}{r}\ \ \ \textup{on $\rn$.}\label{cutoff4}
	\end{eqnarray}
	Given  $n>0, L>1$, we easily see that the function $\left[\thl\right]^n\m\eta^2$ is a legitimate test function for \eqref{e2}. Upon using it, we obtain
	\begin{eqnarray}
	\lefteqn{\frac{1}{2}\frac{d}{dt}\io\int_{0}^{|\m|^2}\left[\theta_{1,L}(s)\right]^nds\eta^2dx+D^2\io\left[\thl\right]^n|\nabla\m|^2\eta^2dx}\nonumber\\
	&&+\frac{nD^2}{2}\io\left[\thl\right]^{n-1}\left|\nabla\thl\right|^2\eta^2dx+\io|\m|^{2\gamma}\left[\thl\right]^n\eta^2dx\nonumber\\
	&=&E^2\io\mnp^2\left[\thl\right]^n\eta^2dx-D^2\io\left[\thl\right]^n\nabla\m\m2\eta\nabla \eta dx\nonumber\\
	&\leq &E^2\io\mnp^2\left[\thl\right]^n\eta^2dx+\frac{D^2}{2}\io\left[\thl\right]^n|\nabla\m|^2\eta^2dx\nonumber\\
	&&+\frac{c}{r^2}\io\left[\thl\right]^n|\m|^2dx.
	\label{10f1}\end{eqnarray}
	Here we have used the fact that
	\begin{equation}\label{10os1}
	\nabla\m\m=\frac{1}{2}\nabla|\m|^2, \ \ \nabla\thl=0\ \ \mbox{on $\{|\m|^2\geq L\}\cup\{|\m|^2\leq 1\}$.}
	\end{equation}
	Integrate \eqref{10f1} to obtain
	\begin{eqnarray}
	\lefteqn{\esup_{0\leq t\leq T}\int_{B_{r}(y)}\int_{0}^{|\m|^2}\left[\theta_{1,L}(s)\right]^nds\eta^2dx+n\int_{\Lambda_r(y)}\left[\thl\right]^{n-1}\left|\nabla\thl\right|^2\eta^2 dxdt}\nonumber\\
	&\leq &c\int_{\Lambda_r(y)}\mnp^2\left[\thl\right]^n\eta^2dxdt+\frac{c}{r^2}\int_{\Lambda_r(y)}\left[\thl\right]^n|\m|^2dxdt\nonumber\\
	&&+\int_{B_{r}(y)}\int_{0}^{|\m_0|^2}\left[\theta_{1,L}(s)\right]^nds\eta^2dx.\label{9t2}
	\end{eqnarray}
	To bound the first term on the right hand side of \eqref{9t2}, we keep \eqref{pbd} in mind and
	%recall a result from \cite{LX} which states that $p\in L^\infty(\ot)$. Then we can 
	use $\left[\thl\right]^n(p-\pyt)\eta^2$ as a test function in \eqref{e1} to derive
	\begin{eqnarray}
	\lefteqn{\int_{B_{r}(y)}\left[\thl\right]^n|\nabla p|^2\eta^2dx+\int_{B_{r}(y)}\left[\thl\right]^n\mnp^2\eta^2dx}\nonumber\\
	&=&-\int_{B_{r}(y)}\nabla p\left[\thl\right]^n(p-\pyt)2\eta\nabla\eta dx\nonumber\\
	&&-n\int_{B_{r}(y)}\nabla p(p-\pyt)\left[\thl\right]^{n-1}\nabla\thl\eta^2dx\nonumber\\
	&&-n\int_{B_{r}(y)}\mnp\m\left[\thl\right]^{n-1}\nabla\thl(p-\pyt)\eta^2dx\nonumber\\
	&&-\int_{B_{r}(y)}\left[\thl\right]^n(p-\pyt)\mnp\m 2\eta\nabla\eta dx\nonumber\\
	&&+\int_{B_{r}(y)}S(x)\left[\thl\right]^n(p-\pyt)\eta^2dx\nonumber\\
	&\leq &\frac{1}{2}\int_{B_{r}(y)}\left[\thl\right]^n|\nabla p|^2\eta^2dx+\frac{c\delta_r^2(y)}{r^2}\int_{B_{r}(y)}\left[\thl\right]^n dx\nonumber\\
	&&+c\delta_r^2(y)n^2\int_{B_{r}(y)}\left[\thl\right]^{n-2}|\nabla\thl|^2\eta^2dx\nonumber\\
	&&+\frac{1}{2}\int_{B_{r}(y)}\left[\thl\right]^n\mnp^2\eta^2dx\nonumber\\
	&&+\frac{c\delta_r^2(y)}{r^2}\int_{B_{r}(y)}|\m|^2\left[\thl\right]^n dx\nonumber\\
	&&+\delta_r(y)\|S(x)\|_{\frac{N}{2},B_{r}(y)}\left(\int_{B_{r}(y)}\left(\left[\thl\right]^n\eta^2\right)^\frac{N}{N-2}dx\right)^{\frac{N-2}{N}}\nonumber\\
	&&+c\delta_r^2(y)n^2\int_{B_{r}(y)}|\m|^2\left[\thl\right]^{n-2}|\nabla\thl|^2\eta^2dx.\label{10f2}
	\end{eqnarray}
	Remember that $\thl\geq 1$. Thus we always have
	\begin{equation}
	\left[\thl\right]^{n-2}\leq \left[\thl\right]^{n-1}.
	\end{equation}
	We can easily see from \eqref{10os1} that
	%Note that
	\begin{eqnarray}
	|\m|^2\left[\thl\right]^{n-2}|\nabla\thl|^2&=&
	\left[\thl\right]^{n-1}|\nabla\thl|^2.
	\end{eqnarray}
	Use the Gagliardo-Nirenberg-Sobolev inequality to estimate
	\begin{eqnarray}
	\lefteqn{\delta_r(y)\|S(x)\|_{\frac{N}{2},B_{r}(y)}\left(\int_{B_{r}(y)}\left(\left[\thl\right]^n\eta^2\right)^\frac{N}{N-2}dx\right)^{\frac{N-2}{N}}}\nonumber\\
	&\leq &\delta_r(y)\|S(x)\|_{\frac{N}{2},B_{r}(y)}\left(\int_{B_{r}(y)}\left(\left[\thl\right]^{\frac{n+1}{2}}\eta\right)^\frac{2N}{N-2}dx\right)^{\frac{N-2}{N}}\nonumber\\
	&\leq &\delta_r(y)\|S(x)\|_{\frac{N}{2},B_{r}(y)}\int_{B_{r}(y)}\left|\nabla\left(\left[\thl\right]^{\frac{n+1}{2}}\eta\right)\right|^2dx.
	%\nonumber\\
	%&\leq &c(n+1)^2\delta_r(y)\|S(x)\|_{\frac{N}{2},B_{r}(y)}\int_{B_{r}(y)}\left[\thl\right]^{n-1}|\nabla\thl|^2\eta^2dx
	\end{eqnarray}
	Keeping those in mind,  we integrate \eqref{10f2} with respect to $t$ over $(0, T)$ and then use \eqref{9t2} in the resulting inequality to derive
	\begin{eqnarray}
	\lefteqn{\int_{\Lambda_r(y)}\left[\thl\right]^n|\nabla p|^2\eta^2dxdt+\int_{\Lambda_r(y)}\left[\thl\right]^n\mnp^2\eta^2dxdt}\nonumber\\
	&\leq &\frac{c\delta_r^2(y)}{r^2}\int_{\Lambda_r(y)}\left[\thl\right]^n dxdt+\frac{c\delta_r^2(y)}{r^2}\int_{\Lambda_r(y)}|\m|^2\left[\thl\right]^n dxdt\nonumber\\
	%&&\nonumber\\
	&&+(n+1)^2\delta_r(y)\|S(x)\|_{\frac{N}{2},B_{r}(y)})\int_{\Lambda_r(y)}\left[\thl\right]^{n-1}|\nabla\thl|^2\eta^2dxdt\nonumber\\
	&&+\frac{c\delta_r(y)\|S(x)\|_{\frac{N}{2},B_{r}(y)}}{r^2}\int_{\Lambda_r(y)}\left[\thl\right]^{n+1} dxdt+c(\delta_r^2(y)n^2\nonumber\\
	%&&+c\delta_r(y)\|S(x)\|_{\frac{N}{2},B_{r}(y)}\int_{B_{r}(y)}\left[\thl\right]^{n-1}|\nabla\thl|^2\eta^2dx\nonumber\\
	&\leq &\frac{c\delta_r^2(y)}{r^2}\int_{\Lambda_r(y)}\left[\thl\right]^n dxdt\nonumber\\
	&&+\frac{c\delta_r^2(y)(n+1)}{r^2}\int_{\Lambda_r(y)}|\m|^2\left[\thl\right]^n dxdt\nonumber\\
	&&+c\delta_r(y)n\int_{\Lambda_r(y)}\mnp^2\left[\thl\right]^n\eta^2dxdt\nonumber\\
	&&+c\delta_r(y)n\int_{B_{r}(y)}\int_{0}^{|\m_0|^2}\left[\theta_{1,L}(s)\right]^nds\eta^2dx\nonumber\\
	&&+\frac{c\delta_r(y)}{r^2}\int_{\Lambda_r(y)}\left[\thl\right]^{n+1} dxdt.\label{10s1}
	\end{eqnarray}
	By \eqref{oscp1}, we can choose $r\in(0,\dy)$ so that
	\begin{equation}
	c\delta_r(y)n\leq \frac{1}{2}.
	\end{equation}
	With this in hand, we can combine \eqref{9t2} with \eqref{10s1} to deduce
	\begin{eqnarray}
	\lefteqn{\esup_{0\leq t\leq T}\int_{B_{r}(y)}\int_{0}^{|\m|^2}\left[\theta_{1,L}(s)\right]^nds\eta^2dx}\nonumber\\
	&&+n\int_{\Lambda_r(y)}\left[\thl\right]^{n-1}\left|\nabla\thl\right|^2\eta^2 dxdt\nonumber\\
	&\leq &\frac{c\delta_r^2(y)}{r^2}\int_{\Lambda_r(y)}\left[\thl\right]^n dxdt+\frac{c}{r^2}\int_{\Lambda_r(y)}\left[\thl\right]^n|\m|^2dxdt\nonumber\\
	&&+\int_{B_{r}(y)}\int_{0}^{|\m_0|^2}\left[\theta_{1,L}(s)\right]^nds\eta^2dx+\frac{c\delta_r(y)}{r^2}\int_{\Lambda_r(y)}\left[\thl\right]^{n+1} dxdt.\label{10os2}
	\end{eqnarray}
	Taking $L\rightarrow\infty$ in \eqref{10os2} yields
	\begin{eqnarray}
	\lefteqn{\esup_{0\leq t\leq T}\int_{B_{r}(y)}|\m|^{2(n+1)}\eta^2dx+\int_{\Lambda_r(y)}\left|\nabla\left(|\m|^{n+1}\eta\right)\right|^2dxdt}\nonumber\\
	&\leq& c(r)\int_{\Lambda_r(y)}|\m|^{2(n+1)}dxdt+c.
	\end{eqnarray}
	Now we are in a position to apply the Gagliardo-Nirenberg-Sobolev inequality to derive
	\begin{eqnarray}
	\lefteqn{\int_{Q_{\frac{r}{2}}(y)}\left(|\m|^{2(n+1)}\right)^{(1+\frac{2}{N})}dxdt}\nonumber\\
	&\leq &\int_{\Lambda_r(y)}\left(|\m|^{n+1}\eta\right)^{(2+\frac{4}{N})}dxdt\nonumber\\
	&\leq &\int_{0}^{T}\left(\int_{B_{r}(y)}\left(|\m|^{n+1}\eta\right)^{\frac{2N}{N-2}}dx\right)^{\frac{N-2}{N}}\left(\int_{B_{r}(y)}|\m|^{2(n+1)}\eta^2dx\right)^{\frac{2}{N}}dt\nonumber\\
	&\leq &\left(\esup_{0\leq t\leq T}\int_{B_{r}(y)}|\m|^{2(n+1)}\eta^2dx\right)^{\frac{2}{N}}\int_{\Lambda_r(y)}\left|\nabla\left(|\m|^{n+1}\eta\right)\right|^2dxdt\nonumber\\
	&\leq &\left(c(r)\int_{\Lambda_r(y)}|\m|^{2(n+1)}dxdt+c\right)^{1+\frac{2}{N}}<\infty.
	\end{eqnarray}
	If $y\in \partial\Omega$, 
	%then we need to replace $B_r(y)$ by $\Omega_r(y)$ throughout. Furthermore,
	the only change you need to make is that 
	in the test function for \eqref{e2} we substitute  $p$ for $p-\pyt$. Everything else is exactly the same.
	% and . 
	%If $N=2$, use the two-dimensional version of the Gagliardo-Nirenberg-Sobolev inequality. 
	This completes the proof.
\end{proof}
\begin{theorem}\label{hconm}Let \textup{(H1)},\textup{(A1)}-\textup{(A3)} be satisfied and $(p,\m)$ be a weak solution of \eqref{e1}-\eqref{e4}. Assume that $\m_0$ is H\"{o}lder continuous on $\oo$.
	%\in \left(C^\alpha()\right)^
	If $p\in L^\infty(0,T; C^\alpha(\oo))$ for $\alpha\in (0,1)$, then $\m\in\left(C^{\beta,\frac{\beta}{2}}(\overline{\ot})\right)^N$ for some $\beta\in (0,1)$.
\end{theorem}
\begin{proof} In view of Lemma 2.2 in \cite{Y}, it is enough for us to show that there exist $c, \beta>0$ such that
	\begin{equation}
	\esup_{0\leq t\leq T}\int_{ B_r(y)}|\f|dx\leq cr^{N-2+2\beta}\ \ \mbox{for all $y\in\oo$ and $r>0$,}
	\end{equation}
	where
	\begin{equation}
	\f=E^2\mnp\nabla p-|\m|^{2\gamma-2}\m.
	\end{equation}
	To this end, let $y$ be given as in the theorem
	% we assume that $y\in\Omega$ 
	and choose a smooth cut-off function $\eta$ as in \eqref{cutoff1}-\eqref{cutoff4}. If $\partial\Omega\cap B_r(y)=\emptyset$, then
	% with $r$ being replaced with $2r$. 
	we use $(p-p_{y,r})\eta^2$ as a test function in \eqref{e1} to obtain
	\begin{eqnarray}
	\lefteqn{\int_{B_{r}(y)}|\nabla p|^2\eta^2dx+	\int_{B_{r}(y)}\mnp^2\eta^2dx}\nonumber\\
	&=&-\int_{B_{r}(y)}\nabla p(p-p_{y,r})2\eta\nabla\eta dx-\int_{B_{r}(y)}\mnp\m(p-p_{y,r})2\eta\nabla\eta dx\nonumber\\
	&&+\ib S(x)(p-p_{y,r})\eta^2dx\nonumber\\
	&\leq &\frac{1}{2}\int_{B_{r}(y)}|\nabla p|^2\eta^2dx+\frac{1}{2}\int_{B_{r}(y)}\mnp^2\eta^2dx\nonumber\\
	&&+cr^{N-2+2\alpha}+cr^{-2+2\alpha}\int_{B_{r}(y)}|\m|^2dx+cr^{N-2+2-\frac{N}{q}+\alpha},
	\end{eqnarray}
	where $q$ is given as in (A1).
	This yields
	\begin{equation}\label{hcon3}
	\int_{B_{r}(y)}|\nabla p|^2\eta^2dx+	\int_{B_{r}(y)}\mnp^2\eta^2dx\leq cr^{N-2+2\alpha}+cr^{-2+2\alpha}\int_{B_{r}(y)}|\m|^2dx.
	\end{equation}
	Here without any loss of generality we have assumed that $2-\frac{N}{q}\geq\alpha$.
	%We can infer from our assumption on $p$ and 
	Theorem \ref{linfty} together with the finite covering theorem  implies that
	\begin{equation}\label{hcon2}
	|\m|^2\in L^\infty(0,T; L^s(\Omega))\ \ \mbox{for each $s\geq 1$.}
	\end{equation}
	Consequently,
	\begin{equation}
	\int_{B_{r}(y)}|\m|^2dx\leq \left(\int_{B_{r}(y)}|\m|^{2s}dx\right)^{\frac{1}{s}}r^{\frac{N(s-1)}{s}}\leq c(s)r^{\frac{N(s-1)}{s}}\ \ \mbox{for each $s>1$.}
	\end{equation}
	Use this in \eqref{hcon3} to derive
	\begin{equation}\label{10ot3}
	\int_{B_{r}(y)}|\nabla p|^2\eta^2dx+	\int_{B_{r}(y)}\mnp^2\eta^2dx\leq cr^{N-2+2\alpha}+cr^{N-2+2\alpha-\frac{N}{s}}
	\end{equation}
	Similarly,
	\begin{equation}
	\int_{ B_r(y)}|\m|^{2\gamma-1}dx\leq \left(\int_{B_{r}(y)}|\m|^{(2\gamma-1)s}dx\right)^{\frac{1}{s}}r^{\frac{N(s-1)}{s}}\leq c(s)r^{\frac{N(s-1)}{s}}=cr^{N-2+2-\frac{N}{s}}.
	\end{equation}
	Choose $s$ so large that $2\alpha-\frac{N}{s}>0$. Then take $\beta= \frac{1}{2}\left(2\alpha-\frac{N}{s}\right)$.  If $\partial\Omega\cap B_r(y)\ne\emptyset$,
	%If $y\in\partial\Omega$,
	we substitute $p$ for $p-\pyr$ in the test function for \eqref{e1} and the subsequent calculations are almost identical. This completes the proof.
\end{proof}
It is known from \cite{LX} that $p\in L^\infty(\ot)$ no matter what the space dimension is.  Local boundedness estimates for $p$ turn out to be a much more delicate issue. From here on in this section we will assume $N=2$, or $3$. We must impose this restriction on the space dimension in order for the Moser-De Giorgi type of arguments to work. For simplicity, we will only consider the case where
\begin{equation}\label{conn}
N=3.
\end{equation} The other case is similar and a little bit simpler.
\begin{theorem}\label{pro5}Let \textup{(H1)}, \textup{(A1)}-\textup{(A3)}, and \eqref{conn} hold
	and $(\m, p)$ a weak solution of \eqref{e1}-\eqref{e4}. Assume that $y\in \Omega$ satisfies \eqref{con1} and \eqref{con2}. 
	Then there is a $\beta>0$ with
	\begin{equation}
	\omega_\rho(y,t)=\textup{osc}_{B_\rho(y)}p\leq c\rho^\beta\ \ \mbox{for all $\rho\in (0,\dy)$.}
	\end{equation}
\end{theorem}
We shall adapt an idea from \cite{X3}. For this purpose we first consider
subsolutions of certain homogeneous elliptic equations. 
\begin{definition}
Let $\m$ be given as in \textup{(D1)}.	We say that $v$ is a subsolution of the equation
	\begin{equation}\label{e1z}
	\imm{v}=0\ \ \mbox{in $\Omega$}
	\end{equation} if:
	\begin{enumerate}
		\item[\textup{(D5)}] $v\in W^{1,2}(\Omega), (\m\cdot\nabla v)\in L^2(\Omega)$;
		\item[\textup{(D6)}] $\io\nabla v\nabla\xi dx+\io (\m\cdot\nabla v)(\m\cdot\nabla\xi)dx\leq 0$ for all 
		$\xi\in W_0^{1,2}(\Omega)$ with $\xi\geq 0$ and	$(\m\cdot\nabla \xi)\in L^2(\Omega)$.
	\end{enumerate}
\end{definition}

%In this section, we shall consider the case where $N=3$.
%assume:
%\begin{enumerate}
%	\item[(H4)] , s_0=6
%\end{enumerate}

\begin{clm}\label{subb} Assume that \eqref{conn} holds and $\m\in L^\infty(0,T; \left(W^{1,2}(\Omega)\right)^3)\cap C([0,T]; \left(L^{2}(\Omega)\right)^3)$. Let $y\in\Omega$ be such that \eqref{con1} and \eqref{con2} hold. If $v$ be a subsolution to \eqref{e1z}, then we can find a positive number $c$ with the property
	\begin{equation}
	\esup_{B_{\frac{r}{2}}(y)} v\leq c\left(\aiy(v^+)^3\right)^{\frac{1}{3}}\ \ \mbox{for $r\in (0,\dy)$}.
	\end{equation}
\end{clm}

\begin{proof}
	
	%Denote by $R$ the distance between $y$ and $\partial\Omega$. 
	Fix a $r\in (0, \dy)$. 
	Set
	\begin{eqnarray}
	r&=&\frac{r}{2}+\frac{r}{2^{j+1}},\ \ j=0, 1,2,\cdots.
	\\
	k_j&=&k-\frac{k}{2^j},\ \ j=0, 1,2,\cdots,
	\end{eqnarray}
	where $k$ is a positive number to be determined.
	Then choose a sequence of smooth functions $\eta_j$ so that
	\begin{eqnarray}
	\eta_j&=&1\ \ \ \textup{on $B_{r_{j+1}}(y)$,}\\
	\eta_j&=&0\ \ \ \textup{outside $B_{r_{j}}(y)$,}\\
	|\nabla\eta_j|&\leq&\frac{c2^j}{r}\ \ \ \textup{on $\rn$,}\ \ \ j=0, 1, \cdots.
	\end{eqnarray}
	Use $\vkp\eta_j^2$ as a test function in \eqref{e1z} to obtain
	\begin{eqnarray}
	\lefteqn{\int_{B_{r_{j}}(y)}|\nabla \vkp|^2\eta_j^2dx+\int_{B_{r_{j}}(y)}\nabla v\vkp 2\eta_j\nabla\eta_jdx}\nonumber\\
	&&+\int_{B_{r_{j}}(y)}(\m\cdot\nabla \vkp)^2\eta_j^2dx\nonumber\\
	&&+\int_{B_{r_{j}}(y)}(\m\cdot\nabla v)\vkp2\eta_j(\m\cdot\nabla\eta_j)dx\leq 0.
	%\nonumber\\
	%&=&\int_{B_{r_{j}}(y)}S(x)(v^+)^n\eta_j^2dx,
	\end{eqnarray}
	from whence follows
	\begin{eqnarray}
	\lefteqn{\int_{B_{r_{j}}(y)}|\nabla \vkp|^2\eta_j^2dx}\nonumber\\
	&\leq& \frac{c4^j}{r^2}\int_{B_{r_{j}}(y)}\left[\vkp\right]^2dx+ \frac{c4^j}{r^2}\int_{B_{r_{j}}(y)}|\m|^2\left[\vkp\right]^2dx.
	%\nonumber\\
	%	&&+c\int_{B_{r_{j}}(y)}S(x)(v^+)^n\eta_j^2dx.
	\label{9t1}
	\end{eqnarray}
	Remember that $N=3$. We deduce from  Poincar\'{e}'s inequality that
	%	We use \eqref{bmo} to estimate the second integral on the right-hand side of \eqref{9t1}.
	\begin{eqnarray}
	\lefteqn{\int_{B_{r_{j}}(y)}|\m|^2\left[\vkp\right]^2dx}\nonumber\\
	&\leq&2\int_{B_{r_{j}}(y)}|\m-\m_{y, r}(t)|^2\left[\vkp\right]^2dx\nonumber\\
	&&+2|\m_{y, r}(t)|^2\int_{B_{r_{j}}(y)}\left[\vkp\right]^2dx\nonumber\\
	&\leq&2\left(\int_{B_{r}(y)}|\m-\m_{y, r}(t)|^{6}dx\right)^{\frac{1}{3}}\left(\int_{B_{r_{j}}(y)}\left[\vkp\right]^3dx\right)^{\frac{2}{3}}\nonumber\\
	&&+2\int_{B_{r_{j}}(y)}\left[\vkp\right]^2dx
	\nonumber\\
	&\leq&\left(c\int_{B_{r}(y)}|\nabla\m|^2dx+cr\right)\left(\int_{B_{r_{j}}(y)}\left[\vkp\right]^3dx\right)^{\frac{2}{3}}\nonumber\\
	%+\int_{B_{r_{j}}(y)}(v^+)^{n+1}dx
	&\leq&cr\left(\int_{B_{r_{j}}(y)}\left[\vkp\right]^3dx\right)^{\frac{2}{3}}
	.\label{9w3}
	\end{eqnarray}
	Here we have used \eqref{con1} and \eqref{con2}.
	Plug  \eqref{9w3} into \eqref{9t1} to derive
	\begin{equation}
	\int_{B_{r_{j}}(y)}|\nabla \vkp|^2\eta_j^2dx\leq \frac{c4^j}{r}\left(\int_{B_{r_{j}}(y)}\left[\vkp\right]^3dx\right)^{\frac{2}{3}}
	\end{equation}
	We compute from the Gagliardo-Nirenberg-Sobolev inequality that
	\begin{eqnarray}
	\left(\int_{B_{r_{j+1}}(y)}\left[\vkp\right]^6dx\right)^{\frac{1}{3}}
	&\leq &
	\left(\int_{B_{r_{j}}(y)}\left[\vkp\eta_j\right]^6dx\right)^{\frac{1}{3}}\nonumber\\
	&\leq &c\int_{B_{r_{j}}(y)}\left|\nabla\left[\vkp\eta_j\right]\right|^2dx\nonumber\\
	&\leq &\frac{c4^j}{r}\left(\int_{B_{r_{j}}(y)}\left[\vkp\right]^3dx\right)^{\frac{2}{3}}.\label{9w4}
	\end{eqnarray}
	Let
	\begin{equation}
	Y_j= \int_{B_{r_{j}}(y)}\left[(v-k_j)^+\right]^3dx.
	\end{equation}
	%	where $\chi=\frac{2}{N-2}>1$.
	Then we have
	\begin{equation}
	Y_j\geq \frac{k^3}{8^{j+1}}|\{v\geq k_{j+1}\}|.
	\end{equation}
	We infer from \eqref{9w4} that
	\begin{eqnarray}
	y_{j+1}&\leq &\left(\int_{B_{r_{j+1}}(y)}\left[\vkp\right]^6dx\right)^{\frac{1}{2}}|\{v\geq k_{j+1}\}|^{\frac{1}{2}}\nonumber\\
	&\leq &\frac{c(16\sqrt{2})^j}{k^{\frac{3}{2}}r^{\frac{3}{2}}}Y_j^{1+\frac{1}{2}}
	\end{eqnarray}
	We are in a position to apply Lemma \ref{ynb}, from whence follows
	\begin{equation}
	\esup_{B_{\frac{r}{2}}(y)}v\leq k=c\left(\aiy(v^+)^3dx\right)^{\frac{1}{3}}.
	\end{equation}
\end{proof}
\begin{clm}\label{hconp}Let the assumptions of Claim \ref{subb} hold.
	%	Assume that $\m\in L^\infty(0,T; \left(W^{1,2}(\Omega)\right)^3)$. Let $y\in\Omega$ be such that \eqref{con1} and \eqref{con2} hold. 
	If $v$ is a weak solution of \eqref{e1z}, then
	there exist $c>0, \alpha\in (0,1)$ with the property	\begin{equation}
	\textup{osc}_{B_\rho(y)}v\leq c\left(\frac{\rho}{r}\right)^\alpha\textup{osc}_{B_r(y)}v\ \ \mbox{for all $0<\rho\leq r$.}
	\end{equation}
\end{clm}
\begin{proof}
	Let  $v, y$ be given as in the theorem. Set
	\begin{eqnarray}
	M(\rho)&=& \esup_{B_\rho(y)}v,\\
	m(\rho)&=& \einf_{B_\rho(y)}v.
	\end{eqnarray}
	We introduce two functions due to Moser \cite{M1}:
	\begin{equation}
	w_1=\ln\frac{M(2\rho)-m(2\rho)}{2(M(2\rho)-v)},\ \ w_2=\ln\frac{M(2\rho)-m(2\rho)}{2(v-m(2\rho))}.
	\end{equation}
	It is easy to verify that both $w_1$ and $w_2$ are subsolutions of \eqref{e1z}. There are only two possibilities: either
	\begin{equation}
	|\{w_1^+=0\}\cap B_\rho(y)|=\left|\left\{v\leq \frac{M(2\rho)+m(2\rho)}{2}\right\}\cap B_\rho(y)\right|\geq \frac{1}{2}\left|B_\rho(y)\right|,
	\end{equation}
	%	Otherwise, we would have
	or
	\begin{equation}
	|\{w_2^+=0\}\cap B_\rho(y)|=\left|\left\{v\geq \frac{M(2\rho)+m(2\rho)}{2}\right\}\cap B_\rho(y)\right|\geq \frac{1}{2}\left|B_\rho(y)\right|.
	\end{equation}
	Assume that the first possibility is in force. This puts us in a position to use formula (7.45) in (\cite{GT}, p.164). Doing so yields
	\begin{equation}\label{10ot2}
	\int_{ B_\rho(y)}|w_1^+|^2dx\leq c\rho^2\int_{ B_\rho(y)}|\nabla w_1^+|^2dx.
	\end{equation}
	Let $\eta $ be a smooth cut-off function given as in \eqref{cutoff1}-\eqref{cutoff4} with $r=2\rho$. We use $\frac{\eta^2}{M(2\rho)-v}$ as a test function in \eqref{e1z} to derive
	\begin{eqnarray}
	\lefteqn{\int_{ B_{2\rho}(y)}\frac{\eta^2}{(M(2\rho)-v)^2}|\nabla v|^2dx+\int_{ B_{2\rho}(y)}\frac{\eta^2}{(M(2\rho)-v)^2}(\m\cdot\nabla v)^2dx}\nonumber\\
	&=&-\int_{ B_{2\rho}(y)}\frac{1}{M(2\rho)-v}\nabla v2\eta\nabla\eta dx-\int_{ B_{2\rho}(y)}\frac{1}{M(2\rho)-v}(\m\cdot\nabla v)\m 2\eta\nabla\eta dx.
	\end{eqnarray}
	It immediately follows that
	\begin{equation}
	\int_{ B_{\rho}(y)}|\nabla w_1|^2dx\leq c\rho+\frac{c}{\rho^2}\int_{ B_{2\rho}(y)}|\m|^2dx.\label{10ot1}
	\end{equation}
	We use \eqref{con1} and \eqref{con2} to estimate
	\begin{eqnarray}
	\int_{ B_{2\rho}(y)}|\m|^2dx&\leq &2\int_{ B_{2\rho}(y)}|\m-\m_{y,2\rho}(t)|^2dx+2|\m_{y,2\rho}(t)|^2\rho^3\nonumber\\
	&\leq &c\rho^2\left(\int_{ B_{2\rho}(y)}|\m-\m_{y,2\rho}(t)|^6dx\right)^{\frac{1}{3}}+c\rho^3\nonumber\\
	&\leq &c\rho^2\int_{ B_{2\rho}(y)}|\nabla \m|^2dx+c\rho^3
	\leq c\rho^3.
	\end{eqnarray}
	This together with \eqref{10ot1} implies
	\begin{equation}
	\int_{ B_{\rho}(y)}|\nabla w_1|^2dx\leq c\rho.
	\end{equation}
	We compute from Claim \ref{subb}, \eqref{10ot2}, and Poincar\'{e}'s inequality that
	\begin{eqnarray}
	\esup_{B_{\rho/2}(y)} w_1&\leq &c\left(\avint_{B_{\rho}(y)}(w_1^+)^3dx\right)^{\frac{1}{3}}\nonumber\\
	&\leq & c\rho\left(\avint_{B_{\rho}(y)}|\nabla w_1^+|^2dx\right)^{\frac{1}{2}}+c\left(\avint_{B_{\rho}(y)}(w_1^+)^2dx\right)^{\frac{1}{2}}\leq c.
	\end{eqnarray}
	By the definition of $w_1$, we have
	\begin{equation}
	\mbox{osc}_{B_{\rho/2}(y)}v=M(\rho/2)-m(\rho/2)\leq \left(1-\frac{1}{2e^c}\right)\mbox{osc}_{B_{\rho}(y)}v	.
	\end{equation}
	If the second possibility holds, we use $w_2$ instead and everything else is the same. 
	Our theorem follows from Lemma 8.23 in (\cite{GT}, p.201).
\end{proof}

We are ready to prove Theorem \ref{pro5}.
\begin{proof}[Proof of Theorem \ref{pro5}]
	Let $y$ be given as in the theorem. Fix a $r\in (0,\dy)$. We decompose $p$ into two functions $v$ and $u$ on $B_{r}(y)$, where $v$ is the weak solution of the boundary value problem 
	\begin{align}
	-\textup{{div}}\left[(I+\m\otimes \m)\nabla v\right]&=0\ \ \ \textup{in $B_{r}(y)$,}\label{do1}\\
	v&=p\ \ \ \textup{on $\partial B_{r}(y)$}\label{do2}
	\end{align}
	and $u=p-v$. Obviously, $u$ satisfies
	\begin{align}
	-\textup{{div}}\left[(I+\m\otimes \m)\nabla u\right]&=S(x)\ \ \ \textup{in $B_{r}(y)$,}\label{vdo1}\\
	u&=0\ \ \ \textup{on $\partial B_{r}(y)$.}\label{vdo2}
	\end{align}
	%	we are in a position to assert from (H1) and 
	As a result, we can apply Proposition 2.1 in \cite{LX} to the above problem. 
	%becomes applicable. %that there is a positive  number $c=c(N, q_0)$ such that
	This yields
	\begin{equation}
	\esup_{B_{r}(y)}|u|\leq cr^{2-\frac{3}{q}}\left(\int_{B_{r}(y)}|S(x)|^{q}dx\right)^{\frac{1}{q}}.
	\end{equation}
	Obviously, we can apply Claim \ref{hconp} to $v$.
	Keeping this in mind, we calculate for $\rho\in(0,r)$ that
	\begin{eqnarray}
	\mbox{osc}_{B_{\rho}(y)}p&\leq &\mbox{osc}_{B_{\rho}(y)}v+\mbox{osc}_{B_{\rho}(y)}u\nonumber\\
	&\leq & c\left(\frac{\rho}{r}\right)^\alpha\mbox{osc}_{B_{\rho}(y)}v+cr^{2-\frac{3}{q}}\left(\int_{B_{r}(y)}|S(x)|^{q}dx\right)^{\frac{1}{q}}\nonumber\\
	&\leq &c\left(\frac{\rho}{r}\right)^\alpha\left(\mbox{osc}_{B_{\rho}(y)}p+2\esup_{B_{r}(y)}|u|\right)+cr^{2-\frac{3}{q}}\nonumber\\
	&\leq &c\left(\frac{\rho}{r}\right)^\alpha\mbox{osc}_{B_{\rho}(y)}p+cr^{2-\frac{3}{q}}.
	\end{eqnarray}
	The theorem follows from Lemma 2.1 in (\cite{G}, p.86).
\end{proof}
%	To estimate $u$,
% is a positive number bigger than $1$.
\begin{proof}[Proof of Theorem \ref{partial}]	Let $y\in \Omega$ be given as in the theorem. Then Theorem \ref{pro5} holds, and so does Theorem \ref{linfty}. To establish Theorem \ref{partial}, it is enough for us to
	%apply . For this purpose, we let
	show that there is a positive number $\beta$ such that
	\begin{equation}\label{hcon1}
	\aiz|\m-\mz|^2dxdt\leq cr^{2\beta}\ \ \mbox{for $r$ sufficiently small,}
	\end{equation}
	where $z=(y,\tau)$. Indeed, if the above inequality holds, we can infer from the proof of Theorem 1.2 in (\cite{G}, p.70) that
	\begin{equation}
	\limsup_{r\rightarrow 0}|\mz|<\infty.
	\end{equation}
	This together with Theorem \ref{pro5} and \eqref{hcon1} implies
	\eqref{lx1}.
	%	To prove \eqref{hcon1}, we first assume that $z=(y,\tau)\in \ot$. 
	
	We can further weaken \eqref{hcon1}. In fact,
	we only need to show that there is a $\sigma\in (0,1)$ such that
	\begin{equation}\label{hcon11}
	\aiz|\m-\mz|^{1+\sigma}dxdt\leq cr^{2\beta}\ \ \mbox{for $r$ sufficiently small .}
	\end{equation}
	This is due to the following estimate
	\begin{eqnarray}
	\aiz|\m-\mz|^{2}dxdt&\leq& \left(\aiz|\m-\mz|^{1+\sigma}dxdt\right)^{\frac{1}{1+\delta}}\left(\aiz|\m-\mz|^{\frac{1+\sigma}{\sigma}}dxdt\right)^{\frac{\sigma}{1+\sigma}}\nonumber\\
	&\leq &cr^{\frac{2\beta}{1+\sigma}-\frac{3\sigma}{1+\sigma}}\left(\int_{ Q_r(z)}|\m|^{\frac{1+\sigma}{\sigma}}dxdt\right)^{\frac{\sigma}{1+\sigma}}
	\end{eqnarray}
	In view of Theorem \ref{linfty}, we can choose $\sigma>0, r>0$ so small that
	\begin{equation}
	\frac{2\beta}{1+\sigma}-\frac{3\sigma}{1+\sigma}>0\ \ \mbox{and}\ \ \int_{ Q_r(z)}|\m|^{\frac{1+\sigma}{\sigma}}dxdt<\infty.
	\end{equation}
	%provided that $r$ is suitably small.
	To prove \eqref{hcon11}, we pick a  $r>0$ so that $Q_{2r}(z)\subset\ot$. We decompose $\m$ on  $Q_r(z)$ as follows: Solve the linear problem
	\begin{eqnarray}
	\partial_t\w-D^2\Delta\w&=& 0\ \ \mbox{in  $Q_r(z)$,}\\
	\w&=&\m\ \ \mbox{on $\partial_pQ_r(z)$,}
	\end{eqnarray}
	where $\partial_pQ_r(z)$ denotes the parabolic boundary of $Q_r(z)$. Let $\n=\m-\w$. Denote by $n_i$ the $i^{\mbox{th}}$ component of $\n$. Then $n_i$ satisfies
	\begin{eqnarray}
	\partial_tn_i-D^2\Delta n_i&=&E^2\mnp\partial_{x_i} p-|\m|^{2\gamma-2}m_i\ \ \mbox{in $Q_r(z)$,}\label{ni1}\\
	n_i&=&0\ \ \mbox{on $\partial_pQ_r(z)$}
	\end{eqnarray}
	By slightly modifying the proof of Claim 1 in \cite{X2}, we conclude that there exist $c>0$, $\delta\in (0,1)$ depending only on $ D, \sigma$ such that
	\begin{equation}
	\int_{Q_{\rho}(z)}|\w-\w_{z,\rho}|^{1+\sigma}dxdt\leq c\left(\frac{\rho}{r}\right)^{5+\delta}\int_{Q_{r}(z)}|\w-\w_{z,r}|^{1+\sigma}dxdt\ \ \mbox{for all $0<\rho\leq r$.}
	\end{equation}
	We proceed to estimate $\n$. Theorem \ref{pro5} combined with Theorem \ref{linfty} and the proof of Theorem \ref{hconm} implies that \eqref{10ot3} still holds. With this in mind, we use $\theta_{-L,L}(n_i)$ as a test function in
	\eqref{ni1} to obtain
	\begin{eqnarray}
	\lefteqn{\maxt\int_{ B_r(y)}\int_{0}^{n_i}\theta_{-L,L}(s)dsdx+D^2\int_{ Q_r(z)}|\nabla\theta_{-L,L}(n_i)|^2dxdt}\nonumber\\
	&\leq& cL\int_{\tau-r^2/2}^{\tau+r^2/2}\int_{ B_r(y)}|\mnp||\nabla p|dxdt+ cL\int_{\tau-r^2/2}^{\tau+r^2/2}\int_{ B_r(y)}|\m|^{2\gamma-1}dxdt\nonumber\\
	&\leq &cLr^{3+2\beta}\ \ \mbox{for some $\beta>0$ and $r$ suitably small.}
	\end{eqnarray}
	We calculate from the Gagliardo-Nirenberg-Sobolev inequality that
	\begin{eqnarray}
	\lefteqn{\int_{\tau-r^2/2}^{\tau+r^2/2}\int_{ B_r(y)}\left|\theta_{-L,L}(n_i)\right|^{2+\frac{4}{3}}dxdt}\nonumber\\&\leq &
	\int_{\tau-r^2/2}^{\tau+r^2/2}\left(\int_{ B_r(y)}\left|\theta_{-L,L}(n_i)\right|^{6}dx\right)^{\frac{1}{3}}\left(\int_{ B_r(y)}\left|\theta_{-L,L}(n_i)\right|^{2}dx\right)^{\frac{2}{3}}dt\nonumber\\
	&\leq &c\left(\maxt\int_{ B_r(y)}\left|\theta_{-L,L}(n_i)\right|^{2}dx\right)^{\frac{2}{3}}\int_{ Q_r(z)}|\nabla\theta_{-L,L}(n_i)|^2dxdt\nonumber\\
	&\leq &cL^{\frac{5}{3}}r^{\frac{5(3+2\beta)}{3}},
	\end{eqnarray}	
	from whence follows
	\begin{equation}
	|\{|n_i|\geq L\}|\leq \frac{cr^{\frac{5(3+2\beta)}{3}}}{L^{\frac{5}{3}}}.
	\end{equation}	
	We infer from \eqref{wlp} that
	\begin{equation}
	\int_{ Q_r(z)}|\n|^{\frac{5}{3}-\varepsilon}\leq \frac{c}{\varepsilon}r^{5+2\beta(\frac{5}{3}-\varepsilon)},\ \ \varepsilon\in(0, \frac{2}{3}).
	\end{equation}
	We pick an $\varepsilon$ so that
	\begin{equation}
	\frac{2}{3}-\varepsilon=\sigma.
	\end{equation}
	For $0<\rho\leq r$ we calculate
	\begin{eqnarray}
	\lefteqn{\int_{Q_{\rho}(z)}|\m-\m_{z,\rho}|^{1+\sigma}dxdt}\nonumber\\&\leq &
	c\int_{Q_{\rho}(z)}|\w-\w_{z,\rho}|^{1+\sigma}dxdt+c\int_{Q_{\rho}(z)}|\n-\n_{z,\rho}|^{1+\sigma}dxdt\nonumber\\
	&\leq& c\left(\frac{\rho}{r}\right)^{5+\delta}\int_{Q_{r}(z)}|\w-\w_{z,r}|^{1+\sigma}dxdt+c\int_{Q_{\rho}(z)}|\n|^{1+\sigma}dxdt\nonumber\\
	&\leq& c\left(\frac{\rho}{r}\right)^{5+\delta}\int_{Q_{r}(z)}|\m-\m_{z,r}|^{1+\sigma}dxdt+r^{5+2\beta(1+\sigma)}.
	\end{eqnarray}
	We conclude \eqref{hcon11} from Lemma 2.1 from (\cite{G}, p.86). The proof is complete.
\end{proof}

\section{Proof of Theorems \ref{locexis} and \ref{larexis}}\label{sec4}

%If $\m=(m_1, \cdots,m_N)^T$ is a vector-valued function, then $$
\begin{proof}[Proof of Theorem \ref{locexis}] For each $\varepsilon>0$ we define
	\begin{equation}
	\ml{\m}{\varepsilon}=(\theta_{-\varepsilon,\varepsilon}(m_1), \cdots,\theta_{-\varepsilon,\varepsilon}(m_N))^T.
	\end{equation} 
	Then we have
	\begin{equation}\label{feps}
	\m_0+\ml{\m-\m_0}{\varepsilon}=\m\ \ \textup{on the set where $|\m-\m_0|\leq \varepsilon$.}
	\end{equation}
	Replace $\m$ by $	\m_0+\ml{\m-\m_0}{\varepsilon}$ in \eqref{e1} and
	write the resulting equation in the form
	\begin{eqnarray}
	-\textup{div}\left[(I+\m_0\otimes\m_0)\nabla p\right]&=&\textup{div}\left(\m_0\otimes\ml{\m-\m_0}{\varepsilon}\nabla p\right)+\textup{div}\left(\ml{\m-\m_0}{\varepsilon}\otimes\m_0\nabla p\right)\nonumber\\
	&&+\textup{div}\left(\ml{\m-\m_0}{\varepsilon}\otimes\ml{\m-\m_0}{\varepsilon}\nabla p\right)+S(x) \ \ \textup{in $\ot$.}\label{lapp}
	\end{eqnarray}
	%It is a rather standard procedure that we can show that 
	\begin{clm}
		If $\varepsilon$ is sufficiently small, then \eqref{lapp}
		coupled with \eqref{e2}-\eqref{e4} has a weak solution satisfying (D4). 
	\end{clm}
	\begin{proof}A solution will be constructed via the Leray-Schauder theorem (\cite{GT}, p.280). For this purpose we define an operator $B$ from $\left(L^\infty(\ot)\right)^N$ into itself as follows: For each $\m\in \left(L^\infty(\ot)\right)^N$ we say $B(\m)=\w$ if $\w$ is the unique solution of the initial
		boundary value problem
		\begin{eqnarray}
		\partial_t\w-D^2\Delta\w&=&E^2((\m_0+\ml{\m-\m_0}{\varepsilon})\cdot\nabla p)\nabla p-|\m|^{2(\gamma-1)}\m\ \ \textup{in $\ot$,}\label{cl1}\\
		\w&=&0\ \ \textup{on $\Sigma_T$,}\\
		\w(x,0)&=&\m_0(x)\ \ \textup{on $\Omega$,}
		\end{eqnarray}
		where $p$ solves \eqref{lapp} coupled with \eqref{e3}. The latter problem has a unique solution if $\varepsilon$ is sufficiently small. To see this, first
		%We proceed
		%to establish a priori estimates. 
		observe that the elliptic coefficients on the left-hand side of \eqref{lapp} are continuous. Therefore, we are in a position to apply Lemma \ref{lqe}, from whence follows that for each $q>1$ there is a positive number c determined only by $q, \m_0, N$, and $\Omega$ such that
		% becomes applicable, and upon using it, we arrive at
		\begin{eqnarray}
		\|\nabla p\|_{q,\Omega} &\leq& c\|\f_\varepsilon\otimes\m_0\nabla p\|_{q,\Omega}+c\|\f_\varepsilon\otimes\f_\varepsilon\nabla p\|_{q,\Omega}+c\|S(x)\|_{\frac{Nq}{N+q, \Omega}}\nonumber\\
		&\leq &c(\varepsilon+\varepsilon^2)\|\nabla p\|_{q,\Omega}+c.\label{cl2}
		\end{eqnarray}
		%for each $q>1$, where $c$ is independent of $\varepsilon$. 
		Now fix a $q>2(1+\frac{N}{2})$.
		We have
		\begin{equation}\label{cl4}
		\|\nabla p\|_q \leq c
		\end{equation}
		if we choose $\varepsilon$ so that the coefficient $c(\varepsilon+\varepsilon^2)$
		% of $\|\nabla p\|_{q,\Omega}$ 
		in \eqref{cl2} is strictly less than $1$.
		%to be sufficiently small.
		From here on we assume that this is the case.
		%This together with 
		Subsequently, Lemma \ref{hcon} becomes applicable to \eqref{cl1}. Upon using it, we obtain  that $\w$ is %bounded, and thus it is 
		H\"{o}lder continuous on $\overline{\ot}$. Therefore, we can claim that $B$ is well-defined, continuous, and precompact. 
		It remains to be seen that there is a positive number $c$ such that
		\begin{equation}\label{cl5}
		\|\m\|_{\infty,\ot}\leq c
		\end{equation}
		for all $\m\in \left(L^\infty(\ot)\right)$ and $\sigma\in (0,1]$ satisfying
		\begin{equation*}
		\m=\sigma B(\m).
		\end{equation*}
		This equation is equivalent to the following problem
		\begin{eqnarray}
		-\textup{div}\left[(I+\m_0\otimes\m_0)\nabla p\right]&=&\textup{div}\left(\m_0\otimes\ml{\m-\m_0}{\varepsilon}\nabla p\right)+\textup{div}\left(\ml{\m-\m_0}{\varepsilon}\otimes\m_0\nabla p\right)\nonumber\\
		&&+\textup{div}\left(\ml{\m-\m_0}{\varepsilon}\otimes\ml{\m-\m_0}{\varepsilon}\nabla p\right)+S(x) \ \ \textup{in $\ot$,}\\
		\partial_t\m-D^2\Delta\m&=&E^2\sigma((\m_0+\ml{\m-\m_0}{\varepsilon})\cdot\nabla p)\nabla p\nonumber\\
		&&-\sigma|\m|^{2(\gamma-1)}\m\ \ \textup{in $\ot$,}\label{cl3}\\
		\m&=&0\ \ \textup{on $\Sigma_T$,}\\
		p&=&0\ \ \textup{on $\Sigma_T$,}\\
		\m(x,0)&=&\sigma\m_0(x)\ \ \textup{on $\Omega$.}
		\end{eqnarray}
		We still have \eqref{cl4}. As a result, the right-hand side
		of \eqref{cl3} is bounded in $L^{\frac{q}{2}}(\ot)$. Recall that
		$\frac{q}{2}>1+\frac{N}{2}$. Hence \eqref{cl5} follows from Lemma \ref{hcon}. This completes the proof of the claim.
	\end{proof}
	%	The preceding a priori estimates indeed imply that the problem \eqref{lapp} and
	%	\eqref{e2}-\eqref{e4} has a weak solution, provided that $\varepsilon$ is small enough.
	
	To continue the proof of Theorem \ref{locexis}, by the H\"{o}lder continuity of $\m$ on $\overline{\ot}$, we can find a positive number $T_0\leq T$ such that
	\begin{equation*}
	|\m(x,t)-\m_0(x)|\leq ct^{\frac{\alpha}{2}}\leq\varepsilon\ \ \textup{on $\Omega_{T_0}$,}
	\end{equation*}
	where $\alpha$ is the H\"{o}lder exponent of $\m$.
	We see from \eqref{feps} that \eqref{lapp} reduces to \eqref{e1} on $\Omega_{T_0}$, where (D4)$^\prime$ holds true. 
	%Thus we can apply Lemma \ref{lqe} to \eqref{e1}
	%on $\Omega_{T_0}$ to obtain
	%	\begin{equation}
	%	\sup_{0\leq t\leq T_0}\|\nabla p\|_{q, \Omega}\leq c\ \ \textup{for each $q>1$.}
	%	\end{equation}	That is, $(\m, p)$ is a classical solution of \eqref{e1}-\eqref{e4}. 
	The proof is complete.
	%that $m$ is continuous
\end{proof}

\begin{proof}[Proof of Theorem \ref{larexis}]
	Fix $T>0$. We construct a sequence of functions $\{(\wk,\ \psk)\}$ on $\ot$ as follows: Set
	%by successively solving
	\begin{equation*}
	\w_0=\m_0.
	\end{equation*}
	The function $p_0$ is the unique solution of the boundary value problem
	\begin{eqnarray}
	-\textup{div}[(I+\m_0\otimes \m_0)\nabla p_0] &=& S(x)\label{pz1}
	\ \ \ \textup{in $\Omega$,}\\
	p_0&=& 0\ \ \ \textup{on $\partial\Omega$.}\label{pzt}
	\end{eqnarray}
	Suppose that $\wko,\ \psko, k=1,2,\cdots,$ are known. We define $\psk$ to be the unique solution of the boundary value problem
	\begin{eqnarray}
	%-\textup{div}\left[(I)\right]
	-\Delta\psk&=&\textup{div}\left[(\wko\cdot\nabla\psko)\wko\right]-\Delta p_0-\textup{div}\left[(\m_0\cdot\nabla p_0)\m_0\right]\ \ \textup{in $\ot$,}\label{agp10}\\
	\psk&=&0\ \ \textup{on $\Sigma_T$,}
	\end{eqnarray}
	while $\wk$ solves the problem
	\begin{eqnarray}
	\partial_t\wk-D^2 \Delta\wk+|\wk|^{2(\gamma-1)}\wk&=&E^2(\wko\cdot\nabla\psko)\nabla\psko\  \  \textup{in $\ot$,}\label{agp11}\\
	\wk&=&0\ \ \textup{on $\Sigma_T$,}\\
	\wk(x,0)&=& \m_0(x)\   \  \textup{on $\Omega$.}\label{agp12}
	\end{eqnarray}
	The uniqueness of a solution to the preceding problem can easily be inferred from Lemma \ref{plap}. Obviously, if $\{(\wko,\psko)\}$ satisfies (D4), so does $\{(\wk,\psk)\}$. The sequence $\{(\wk,\psk)\}$ is well-defined.
	%We conclude 
	It follows from Lemma \ref{lwb} that there is a positive number $c=c(N,\Omega)$ with 
	\begin{eqnarray}
	a_k\equiv \|\wk\|_{\infty,\ot}&\leq& c\left(\|\m_0\|_{\infty,\Omega}+T^{\frac{1}{N+2}}\sup_{0\leq t\leq T}\left\||\wko||\nabla\psko|^2\right\|_{N,\Omega}\right)\nonumber\\
	&\leq& c\|\m_0\|_{\infty,\Omega}+cT^{\frac{1}{N+2}} a_{k-1}b_{k-1}^2,\label{ak1}
	\end{eqnarray}
	where
	\begin{equation*}
	b_k=\sup_{0\leq t\leq T}\|\nabla\psk\|_{2N,\Omega}.
	\end{equation*}
	On the other hand, we can deduce from Lemma \ref{lqe} that 
	%For each $q>1$ 
	there is a positive number $c=(N, \Omega)$ such that
	\begin{equation*}
	\|\nabla\psk\|_{2N,\Omega}\leq c\nwko^2\|\nabla\psko\|_{2N,\Omega}+c\|\nabla p_0\|_{2N,\Omega}.
	\end{equation*}
	It immediately follows
	\begin{equation}\label{bk1}
	b_k\leq ca_{k-1}^2b_{k-1}+c\|\nabla p_0\|_{2N, \Omega}.
	\end{equation}
	Define
	\begin{equation*}
	d_k=a_k+b_k.
	\end{equation*}
	Adding \eqref{bk1} to \eqref{ak1}, we derive
	\begin{equation}\label{db1}
	d_k\leq c\left(1+T^{\frac{1}{N+2}}\right)d_{k-1}^3+cd_0.
	\end{equation}
	Observe from \eqref{pz1}-\eqref{pzt} that
	\begin{equation*}
	\|\nabla p_0\|_{2N,\Omega}\leq c\|S(x)\|_{\frac{2N}{3},\Omega}.
	\end{equation*}
	In view of Lemma \ref{small}, if
	\begin{equation*}
	cd_0^2\left(1+T^{\frac{1}{N+2}}\right)\leq c\left(\|\m_0\|_{\infty, \Omega}+\|S(x\|_{\frac{2N}{3},\Omega}\right)^2\left(1+T^{\frac{1}{N+2}}\right)<1
	\end{equation*} then
	\begin{equation}\label{wpb}
	d_k=\|\wk\|_{\infty,\ot}+\sup_{0\leq t\leq T}\|\nabla\psk\|_{2N,\Omega}\leq c\left(\|\m_0\|_{\infty, \Omega}+\|S(x\|_{\frac{2N}{3},\Omega}\right)\equiv C_0.
	\end{equation}
	%\begin{remark}
	%	From the proof of Lemma \ref{lwb}, we see that $|\ot|$ in %\eqref{db1} can be replaced by $\frac{1}{M}\ioT|\wk|dxdt$.
	%	Thus if we could design an approximation scheme in which
	%\end{remark}
	We must show that the whole sequence $\{\wk, \psk\}$ converges in a suitable sense. To this end, we conclude from \eqref{agp11} that
	\begin{eqnarray}
	\lefteqn{\partial_t(\wk-\wko)-D^2 \Delta(\wk-\wko)+|\wk|^{2(\gamma-1)}\wk-|\wko|^{2(\gamma-1)}\wko}\nonumber\\
	&=&E^2\left[(\wko\cdot\nabla\psko)\nabla\psko-(\w_{k-2}\cdot\nabla p_{k-2})\nabla p_{k-2}\right]\  \  \textup{in $\ot$,}\nonumber\\
	&& k=2,3,\cdots.\label{2cl1}
	\end{eqnarray}
	By Lemma \ref{plap}, we have
	\begin{equation*}
	\left(|\wk|^{2(\gamma-1)}\wk-|\wko|^{2(\gamma-1)}\wko\right)\cdot\left(\wk-\wko\right)\geq 0.
	\end{equation*}
	Use $\wk-\wko$ as a test function in \eqref{2cl1} and keep the above inequality and \eqref{wpb} in mind to derive
	\begin{eqnarray}
	\lefteqn{\frac{1}{2}\frac{d}{dt}\io|\wk-\wko|^2dx+D^2\io|\nabla(\wk-\wko)|^2dx}\nonumber\\
	&\leq &E^2\io\left[(\wko\cdot\nabla\psko)\nabla\psko-(\w_{k-2}\cdot\nabla p_{k-2})\nabla p_{k-2}\right](\wk-\wko)dx\nonumber\\
	&\leq& c\left(\io\left|(\wko\cdot\nabla\psko)\nabla\psko-(\w_{k-2}\cdot\nabla p_{k-2})\nabla p_{k-2}\right|^{\frac{2N}{N+2}}dx\right)^{\frac{N+2}{N}}\nonumber\\
	&&+\frac{D^2}{2}\io|\nabla(\wk-\wko)|^2dx.\label{2cl2}
	\end{eqnarray}
	We write
	\begin{eqnarray}
	\lefteqn{(\wko\cdot\nabla\psko)\nabla\psko-(\w_{k-2}\cdot\nabla p_{k-2})\nabla p_{k-2}}\nonumber\\
	&=&\left((\wko-\w_{k-2})\cdot\nabla\psko\right)\nabla\psko+\left(\w_{k-2}\cdot\left(\nabla\psko-\nabla p_{k-2}\right)\right)\nabla\psko\nonumber\\
	&&+\left(\w_{k-2}\cdot\nabla p_{k-2}\right)(\nabla\psko-\nabla p_{k-2}).
	\end{eqnarray}
	Use this in \eqref{2cl2} to obtain
	\begin{eqnarray}
	\lefteqn{\frac{d}{dt}\io|\wk-\wko|^2dx+\io|\nabla(\wk-\wko)|^2dx}\nonumber\\
	&\leq &c\left(\io|\wko-\w_{k-2}|^{\frac{2N}{N-2}}dx\right)^{\frac{N-2}{N}}\left(\io|\nabla\psko|^{N}dx\right)^{\frac{4}{N}}\nonumber\\
	&&+c\|\w_{k-2}\|_{\infty,\ot}^2\left(\|\nabla\psko\|_{2N,\Omega}^2+\|\nabla p_{k-2}\|_{2N,\Omega}^2\right)\nonumber\\
	&&\cdot\left(\io|\nabla\psko-\nabla p_{k-2}|^{\frac{2N}{N+1}}dx\right)^{\frac{N+1}{N}}
	\nonumber\\
	&\leq& cC_0^4\left(\io|\nabla(\wko-\w_{k-2})|^{2}dx+\io\left|\nabla\psko-\nabla p_{k-2}\right|^{2}dx\right).\label{2cl3}
	\end{eqnarray}
	By \eqref{agp10}, we have
	\begin{eqnarray}
	\lefteqn{-\Delta(\psk-\psko)}\nonumber\\
	&=&\textup{div}\left[(\wko\cdot\nabla\psko)\wko-(\w_{k-2}\cdot\nabla p_{k-2}) \w_{k-2}\right]\ \ \textup{in $\ot$,}\nonumber\\
	&&\ \ k=2,3,\cdots.
	\end{eqnarray}
	Upon using $\psk-\psko$ as a test function in the above equation,
	we arrive at
	\begin{eqnarray}
	\lefteqn{\io|\nabla(\psk-\psko)|^2dx}\nonumber\\
	&=&\io\left[(\wko\cdot\nabla\psko)\wko-(\w_{k-2}\cdot\nabla p_{k-2}) \w_{k-2}\right]\nabla(\psk-\psko)dx\nonumber\\
	&\leq&\frac{1}{2}\io\left|(\wko\cdot\nabla\psko)\wko-(\w_{k-2}\cdot\nabla p_{k-2}) \w_{k-2}\right|^{2}dx\nonumber\\
	&&+\frac{1}{2}\io|\nabla(\psk-\psko)|^2dx.\label{2cl4}
	\end{eqnarray}
	We represent
	\begin{eqnarray}
	\lefteqn{(\wko\cdot\nabla\psko)\wko-(\w_{k-2}\cdot\nabla p_{k-2}) \w_{k-2}}\nonumber\\
	&=&((\wko-\w_{k-2})\cdot\nabla\psko)\wko+(\w_{k-2}\cdot(\nabla\psko-\nabla p_{k-2}))\wko\nonumber\\
	&&+(\w_{k-2}\cdot\nabla p_{k-2}) (\wko-\w_{k-2}).
	\end{eqnarray}
	We calculate
	\begin{eqnarray}
	\lefteqn{\io\left|((\wko-\w_{k-2})\cdot\nabla\psko)\wko\right|^2dx}\nonumber\\
	&\leq&\|\wko\|_{\infty,\ot}^2\left(\io\left|\wko-\w_{k-2}\right|^{\frac{2N}{N-2}}dx\right)^{\frac{N-2}{N}}\left(\io\left|\nabla\psko\right|^Ndx\right)^{\frac{2}{N}}\nonumber\\
	&\leq&cC_0^4\io\left|\nabla(\wko-\w_{k-2})\right|^{2}dx.
	\end{eqnarray}
	Similarly, we have
	\begin{equation}
	\io\left|(\w_{k-2}\cdot\nabla p_{k-2}) (\wko-\w_{k-2})\right|^2dx
	\leq cC_0^4\io\left|\nabla(\wko-\w_{k-2})\right|^{2}dx.
	\end{equation}
	Plug the preceding estimates into \eqref{2cl4} to derive
	\begin{eqnarray}
	\lefteqn{\io|\nabla(\psk-\psko)|^2dx}\nonumber\\
	&\leq& cC_0^4\left(\io|\nabla(\wko-\w_{k-2})|^{2}dx+\io\left|\nabla\psko-\nabla p_{k-2}\right|^{2}dx\right).\label{2cl5}
	\end{eqnarray}
	Let
	\begin{equation}
	\eta_k=\ioT|\nabla(\wk-\w_{k-1})|^{2}dxdt+\ioT\left|\nabla\psk-\nabla p_{k-1}\right|^{2}dxdt.
	\end{equation}
	Add \eqref{2cl5} to \eqref{2cl3} and integrate the resulting equation over $(0,T)$ to yield
	\begin{equation}
	\eta_k\leq cC_0^4\eta_{k-1}.
	\end{equation}
	This implies
	\begin{equation}
	\eta_k\leq \left( cC_0^4\right)^{k-1}\eta_1.
	\end{equation}
	Hence if
	\begin{equation}\label{cz}
	cC_0^4<1,
	\end{equation}
	then the two series's
	\begin{eqnarray}
	\nabla\w_0+\nabla\w_1-\nabla\w_0+&\cdots&+\nabla\wk-\nabla\wko+\cdots\ \ \textup{and}\\
	\nabla p_0+\nabla p_1-\nabla p_0+&\cdots&+\nabla\psk-\nabla\psko+\cdots
	\end{eqnarray}
	converge in $L^2(0,T; \left(W^{1,2}(\Omega)\right)^N)$ and $L^2(0,T; (W^{1,2}(\Omega))$, respectively.
	%$\lim_{k\rightarrow\infty}\eta_k=0$. 
	It immediately follows that the two sequences $\{\wk\}$ and $\{\psk\}$ also converge  in 
	$L^2(0,T;\left(W^{1,2}(\Omega)\right)^N)$ and $L^2(0,T;W^{1,2}(\Omega))$, respectively. We can also deduce from
	\eqref{wpb} and Lemma \ref{hcon} that $\{\wk\}$ is uniformly convergent on $\overline{\ot}$. We can let $k\rightarrow\infty$ in 
	\eqref{agp10} and \eqref{agp11}. Note from \eqref{wpb} that \eqref{cz} is valid if we make the term $\|\m_0\|_{\infty,\Omega}+
	\|S(x)\|_{\frac{2N}{3},\Omega}$ suitably small. The proof is complete.
\end{proof}

\end{document}